\documentclass[10pt]{amsart}
\usepackage{times,amsmath,amsbsy,amssymb,amscd,mathrsfs}
\usepackage{graphicx,subfigure,epstopdf,wrapfig,chemarrow}

\usepackage{algorithm2e} 
\usepackage{multicol,multirow}
\usepackage{mathtools}
\usepackage[usenames,dvipsnames,svgnames,table]{xcolor}
\usepackage[all]{xy}
\usepackage{wrapfig}
\usepackage{tcolorbox}

\usepackage{tikz,tikz-cd}
\usepackage[utf8]{inputenc}
\usepackage{pgfplots} 
\usepackage{pgfgantt}
\usepackage{pdflscape}
\pgfplotsset{compat=newest} 
\pgfplotsset{plot coordinates/math parser=false}
\newlength\fwidth

\definecolor{myBlue}{rgb}{0.0,0.0,0.55}
\usepackage[pdftex,colorlinks=true,citecolor=myBlue,linkcolor=myBlue]{hyperref}

\usepackage[hyperpageref]{backref}

\usepackage{comment,enumerate,multicol,xspace}

  \newcounter{mnote}
  \let\oldmarginpar\marginpar
    \renewcommand\marginpar[1]{\-\oldmarginpar[\raggedleft\footnotesize #1]%
    {\raggedright\footnotesize #1}}

\newtheorem{theorem}{Theorem}[section]
\newtheorem{lemma}[theorem]{Lemma}
\newtheorem{corollary}[theorem]{Corollary}
\newtheorem{proposition}[theorem]{Proposition}

\newtheorem{remark}[theorem]{Remark}

\newcommand{\dx}{\,{\rm d}x}
\newcommand{\dd}{\,{\rm d}}

\newcommand{\bs}{\boldsymbol}

\newcommand{\supp}{\operatorname{supp}}

\renewcommand{\div}{\operatorname{div}}
\newcommand{\tr}{\operatorname{tr}}
\newcommand{\dev}{\operatorname{dev}}
\newcommand{\sym}{\operatorname{sym}}
\newcommand{\skw}{\operatorname{skw}}

\newcommand{\step}[1]{\noindent\raisebox{1.5pt}[10pt][0pt]{\tiny\framebox{$#1$}}\xspace}

\newcommand{\vertiii}[1]{{\left\vert\kern-0.25ex\left\vert\kern-0.25ex\left\vert #1 
    \right\vert\kern-0.25ex\right\vert\kern-0.25ex\right\vert}}

\newcommand{\Oplus}{\ensuremath{\vcenter{\hbox{\scalebox{1.5}{$\oplus$}}}}}
\usepackage{booktabs}
\newcommand{\TT}{\texttt{T}}
\newcommand{\TS}{\texttt{S}}

\begin{document}
\title[Hybridizable Symmetric Stress Element]{Hybridizable Symmetric stress elements on the barycentric refinement in Arbitrary Dimensions}
\author{Long Chen}%
 \address{Department of Mathematics, University of California at Irvine, Irvine, CA 92697, USA}%
 \email{chenlong@math.uci.edu}%
 \author{Xuehai Huang}%
 \address{Corresponding author. School of Mathematics, Shanghai University of Finance and Economics, Shanghai 200433, China}%
 \email{huang.xuehai@sufe.edu.cn}%

 \thanks{The first author was supported by NSF DMS-2309777 and DMS-2309785. The second author was supported by the National Natural Science Foundation of China Project 12171300.}

\makeatletter
\@namedef{subjclassname@2020}{\textup{2020} Mathematics Subject Classification}
\makeatother

\maketitle


\begin{abstract}
Hybridizable \(H(\div)\)-conforming finite elements for symmetric tensors on simplices with barycentric refinement are developed in this work for arbitrary dimensions and any polynomial order. By employing barycentric refinement and an intrinsic tangential-normal (\(t\)-\(n\)) decomposition, novel basis functions are constructed to redistribute degrees of freedom while preserving \(H(\div)\)-conformity and symmetry, and ensuring inf-sup stability. These hybridizable elements enhance computational flexibility and efficiency, with applications to mixed finite element methods for linear elasticity.
\end{abstract}

\section{Introduction}
In this paper, we construct hybridizable $H(\div)$-conforming finite elements for symmetric tensors on simplices in arbitrary dimensions. These elements play a critical role in mixed finite element methods for the stress-displacement (Hellinger-Reissner) formulation of the elasticity system. Several finite elements have been developed in the literature~\cite{ArnoldWinther2002,Adams;Cockburn:2005Finite,ArnoldAwanouWinther2008,HuZhang2016,Hu2015a,HuZhang2015,ChenHuang2022,ChenHuang2024,HuangZhangZhouZhu2024}. However, a common characteristic of all these elements is the presence of vertex degrees of freedom (DoFs), which inherently makes them non-hybridizable. 

To address this limitation, we use the barycentric refinement of a simplicial mesh, also called the Alfeld split~\cite{LaiSchuma2007Spline}. Let \(\texttt{v}_{c}\) be the barycenter of the \(d\)-dimensional simplex \(T\).  
Connecting $\texttt{v}_{c}$ to the vertices of $T$ divides \(T\) into \(d+1\) smaller simplices, each with the barycenter \(\texttt{v}_c\) as a common vertex. We call $T$ the coarse element, and denote the barycentric split by $T^{\rm R}$.

To eliminate vertex DoFs, hybridizable $H(\div)$-conforming symmetric stress elements on the barycentric refinement in two dimensions were developed in~\cite{JohnsonMercier1978,ArnoldDouglasGupta1984,ChristiansenHu2023}. The lowest-order hybridizable $H(\div)$-conforming symmetric stress elements on the barycentric refinement were proposed in~\cite{Krizek1982} for three dimensions and recently extended to arbitrary dimensions in~\cite{GopalakrishnanGuzmanLee2024}. Further reduced stress elements have been developed in two dimensions \cite{ChristiansenHu2023} and in three dimensions \cite{GopalakrishnanGuzmanLee2024}. However, the $H(\div)$-conforming symmetric stress elements on the barycentric refinement in~\cite{ChristiansenGopalakrishnanGuzmanHu2024} still involve vertex DoFs. Hybridizable $H(\div)$-conforming symmetric stress elements have also been developed on the Worsey-Farin split in three dimensions, dividing each tetrahedron into twelve sub-tetrahedra~\cite{GongGopalakrishnanGuzmanNeilan2023}. Hybridizable symmetric stress elements with rational shape functions were given in~\cite{GuzmanNeilan2014}, while hybridizable virtual elements for symmetric tensors were discussed in~\cite{DassiLovadinaVisinoni2020}. A hybridizable elasticity element method was also developed in~\cite{GongWuXu2019}, whose stability depends on the stability of the Scott-Vogelius element~\cite{ScottVogelius1985,ArnoldQin1992,Zhang2005} for the Stokes equation on some special grids.

In Table~\ref{tab:barycentric-elements}, we summarize the existing hybridizable elasticity elements on barycentric subdivision and the finite elements developed in this paper. The table shows the stress and displacement spaces, the number of local DoFs for stress, and the convergence rates of $\|\boldsymbol{\sigma}-\boldsymbol{\sigma}_h\|$. These elements can be classified into two categories: a pair defined on the split mesh $T^{\rm R}$ or on the coarse mesh $T$, with the latter viewed as a macro element. The existing elements include the Johnson-Mercier (JM) element~\cite{JohnsonMercier1978}, the Arnold-Douglas-Gupta (ADG) element~\cite{ArnoldDouglasGupta1984}, the Christiansen-Hu (CH) element~\cite{ChristiansenHu2023}, the K\v{r}\'i\v{z}ek element~\cite{Krizek1982}, and the Gopalakrishnan-Guzman-Lee (GGL) element~\cite{GopalakrishnanGuzmanLee2024}.
The new elements in this paper are the spaces $\Sigma_{\mathrm{RM}}^{\div}(T; \mathbb{S})$ in \eqref{eq:SigmaRM}, $\Sigma_{k,\phi}^{\div}(T^{\rm R}; \mathbb{S})$ in \eqref{eq:localsigmadivk}, $\Sigma_{k,\phi,nn}^{\div}(T^{\rm R}; \mathbb{S})$ in \eqref{eq:localsigmadivknn}, and $\Sigma_{k,\psi}^{\div}(T; \mathbb{S})$ in \eqref{eq:Sigmakpsi}. The pair $\Sigma_{k,\phi}^{\div}(T^{\rm R}; \mathbb{S})$--$\mathbb{P}_{k-1}^{-1}(T^{\rm R}; \mathbb{R}^d)$ is stable on the refined mesh $T^{\rm R}$, 
while $\Sigma_{k,\psi}^{\div}(T; \mathbb{S})$--$\mathbb{P}_{k-1}^{-1}(T; \mathbb{R}^d)$ is a macroelement pair on the coarse mesh. We also use our notation for existing elements of the same type, though they are not identical. A more precise definition of these spaces can be found in the corresponding references.

\begin{table}[htbp]
  \centering
  \caption{Elasticity elements on barycentric subdivisions for $k\geq 2$.}
  \label{tab:barycentric-elements}
  \renewcommand{\arraystretch}{1.25}
    \begin{tabular}{@{} c c c c c c @{}}
      \toprule
      \textbf{Refs} & \textbf{$\mathbb R^d$} & \textbf{Stress Element} & \textbf{\# DoFs for Stress} & \textbf{Displacement} & $\|\boldsymbol{\sigma} - \boldsymbol{\sigma}_h\|$ \\
      \midrule
JM~\cite{JohnsonMercier1978} & 2D & $\Sigma_{1,\phi}^{\div}(T^{\rm R}; \mathbb{S})$ & $15$ & $\mathbb{P}_{1}(T;\mathbb{R}^2)$ & $h^2$ \\
ADG~\cite{ArnoldDouglasGupta1984} & 2D & $\Sigma_{k,\phi}^{\div}(T^{\rm R}; \mathbb{S})$ & $\tfrac{3}{2}(3k^2 + 5k + 2)$ & $\mathbb{P}_{k-1}^{-1}(T^{\rm R};\mathbb{R}^2)$ & $h^{k+1}$ \\
ADG~\cite{ArnoldDouglasGupta1984} & 2D & $\Sigma_{k,\psi}^{\div}(T; \mathbb{S})$ & $\tfrac{3}{2}(k^2 + 3k + 4)$ & $\mathbb{P}_{k-1}(T;\mathbb{R}^2)$ & $h^{k+1}$ \\
CH~\cite{ChristiansenHu2023} & 2D & Reduced space & $9$ & $\mathbb{P}_{0}^{\div}(T^{\rm R};\mathbb{R}^2)$ & $h$ \\
K\v{r}\'i\v{z}ek~\cite{Krizek1982} & 3D & $\Sigma_{1,\phi}^{\div}(T^{\rm R}; \mathbb{S})$ & $42$ & $\mathbb{P}_{1}(T;\mathbb{R}^3)$ & $h^2$ \\
GGL~\cite{GopalakrishnanGuzmanLee2024} & 3D & $\Sigma_{1,\phi}^{\div}(T^{\rm R}; \mathbb{S})$ & $42$ & $\mathbb{P}_{0}^{-1}(T^{\rm R}; \mathbb{R}^d)$ & $h^2$ \\
GGL~\cite{GopalakrishnanGuzmanLee2024} & 3D & Reduced space & $24$ & $\mathrm{RM}(T)$ & $h$ \\
GGL~\cite{GopalakrishnanGuzmanLee2024} & $d$D & $\Sigma_{1,\phi}^{\div}(T^{\rm R}; \mathbb{S})$ & $\tfrac{1}{2}(d^2 + d)(2d + 1)$ & $\mathbb{P}_{1}(T;\mathbb{R}^d)$ & $h^2$ \\
      New & $d$D & $\Sigma_{\mathrm{RM}}^{\div}(T; \mathbb{S})$ & $\tfrac{1}{2}d(d+1)^2$ & $\mathrm{RM}(T)$ & $h$ \\
      New & $d$D & $\Sigma_{k,\phi}^{\div}(T^{\rm R}; \mathbb{S})$ & \eqref{eq:dimSigmaDivk} & $\mathbb{P}_{k-1}^{-1}(T^{\rm R};\mathbb{R}^d)$ & $h^{k+1}$ \\
      New & $d$D & $\Sigma_{k,\phi,nn}^{\div}(T^{\rm R}; \mathbb{S})$ & \eqref{eq:dimSigmaDivknn} & $\mathbb{P}_{k-1}^{-1}(T^{\rm R}; \mathbb{R}^d)$ & $h^{k+1}$ \\
      New & $d$D & $\Sigma_{k,\psi}^{\div}(T; \mathbb{S})$ & \eqref{eq:dimSigmaDivkpsi} & $\mathbb{P}_{k-1}^{-1}(T;\mathbb{R}^d)$ & $h^{k+1}$ \\
      \bottomrule
    \end{tabular}%
\end{table}

The Arnold-Douglas-Gupta element~\cite{ArnoldDouglasGupta1984} covers all $k \geq 1$ but is limited to $d = 2$, while the Gopalakrishnan-Guzm\'{a}n-Lee element~\cite{GopalakrishnanGuzmanLee2024} applies to arbitrary $d$ but is restricted to $k = 1$. Our contribution is the construction of hybridizable $H(\div)$-conforming symmetric stress elements on the barycentric refinement for any polynomial order $k \geq 1$ in arbitrary dimensions $\mathbb{R}^d$ with $d \geq 2$. For the first-order ($h$) element, ours are $9$ ($d=2$) and $24$ ($d=3$), matching the dimension of the reduced spaces in~\cite{ChristiansenHu2023,GopalakrishnanGuzmanLee2024}. 

We employ the tangential-normal ($t$-$n$) decomposition framework developed in our recent work~\cite{ChenHuang2024}.
Specifically, the polynomial space of symmetric tensors on a simplex $T$ of degree $k$ can be expressed as:
\[
\mathbb{P}_k(T; \mathbb{S}) = \Oplus_{\ell = 0}^d \Oplus_{f \in \Delta_{\ell}(T)} \left[ \mathbb{B}_k \mathscr{T}^f(\mathbb{S}) \oplus \mathbb{B}_k \mathscr{N}^f(\mathbb{S}) \right],
\]
where $\mathbb{S} = \mathscr{T}^f(\mathbb{S}) \oplus \mathscr{N}^f(\mathbb{S})$ is a tangential-normal decomposition of the symmetric tensor space $\mathbb S$ and $\mathbb{B}_k \mathscr{T}^f(\mathbb{S}) = b_f\mathbb{P}_{k-(\ell+1)}(f)\otimes \mathscr{T}^f(\mathbb{S})$, with $b_f$ being the $(\ell+1)$th degree bubble function associated with the subsimplex \(f\). 
The tangential component $\mathbb{B}_k \mathscr{T}^f(\mathbb{S})$ contributes to the div bubble space which can be determined by DoFs interior to $T$.

The normal component $\mathbb{B}_k \mathscr{N}^f(\mathbb{S})$ will determine the trace on $\partial T$. To impose the symmetry constraint on the normal tensor space $\mathscr{N}^f \otimes \mathscr{N}^f$, a global basis for the normal plane $\mathscr{N}^f$ over subsimplices of dimensions $0, 1, \ldots, d-2$ is usually used in existing construction. For off-diagonal components $\sym(\bs{n}_{F_i} \otimes \bs{n}_{F_j}), i \neq j$, the symmetry restriction ensures that these components can only be distributed to either face $F_i$ or $F_j$ but not both, resulting in a missing lower or upper triangular part. 

Using the barycentric refinement of the simplex, we construct an $H(\div)$-conforming piecewise polynomial element at a subsimplex \(f \in \Delta_{\ell}(T)\) for \(\texttt{v}_i, \texttt{v}_j \not\in f\) with \(i < j\) as:
\[
b^{\rm R}_f \boldsymbol{\phi}_{ij}^{f} := b^{\rm R}_f \left[ \chi_{T_{i}} \sym(\boldsymbol{t}_{f(0), c} \otimes \boldsymbol{t}_{f(0), j}) - \chi_{T_{j}} \sym(\boldsymbol{t}_{f(0), c} \otimes \boldsymbol{t}_{f(0), i}) \right],
\]
where $b^{\rm R}_f$ denotes the bubble polynomial associated with $f$ on $T^{\rm R}$, and $\chi_{T_i}$ and $\chi_{T_j}$ are the characteristic functions of $T_i$ and $T_j$, respectively.

Denote:
\[
\mathbb{B}_k \Phi^f(\mathbb{S}) = \mathbb{P}_{k-(\ell+1)}(f) \otimes {\rm span} \{ b^{\rm R}_f \boldsymbol{\phi}_{ij}^{f}, \texttt{v}_i, \texttt{v}_j \not\in f, i < j \}.
\]
We enrich the polynomial space on $T$ by the piecewise polynomial space $\mathbb{B}_k \Phi^f(\mathbb{S})$ on the barycentric refinement $T^{\rm R}$, and define
\begin{align*}
\Sigma_{k,\phi}^{\operatorname{div}}(\mathcal{T}_h; \mathbb{S}) &= \{\boldsymbol{\tau}_h \in H(\div, \Omega; \mathbb{S}) : \\
&\qquad \boldsymbol{\tau}_h|_T \in \mathbb{P}_k(T; \mathbb{S}) \oplus \Oplus_{\ell = 0}^{d-2} \Oplus_{f \in \Delta_{\ell}(T)} \mathbb{B}_k \Phi^f(\mathbb{S}) \text{ for } T \in \mathcal{T}_h \}.
\end{align*}
This enrichment  leads to the facewise redistribution of DoFs, ensuring the hybridization capability of the element. 
The following DoFs determine $\Sigma_{k,\phi}^{\operatorname{div}}(\mathcal{T}_h; \mathbb{S})$, for $k \geq 1$:
\begin{equation}\label{eq:faceDoF}
\int_F \bs{\tau} \bs{n}_F \cdot \bs{q} \dd s, \quad F \in \Delta_{d-1}(\mathcal{T}_h), \bs{q} \in \mathbb{P}_{k}(F; \mathbb{R}^d),
\end{equation}
\[
\int_T \bs{\tau} : \bs{q} \dx, \quad T \in \Delta_d(\mathcal{T}_h), \; \bs{q} \in \mathbb{P}_{k-2}(T; \mathbb{S}).
\]
The facewise DoFs \eqref{eq:faceDoF} enable hybridization~\cite{fraeijs1965displacement,ArnoldBrezzi1985}, relaxing the normal continuity of the stress element via Lagrange multipliers.

In view of the face DoF \eqref{eq:faceDoF}, the element $\Sigma_{k, \phi}^{\operatorname{div}}(\mathcal T_h; \mathbb{S})$ is the generalization of Brezzi-Douglas-Marini/N\'ed\'elec (2nd kind) div-conforming vector element~\cite{BrezziDouglasMarini1985,Nedelec:1986family,BrezziDouglasDuranFortin1987} to div-conforming symmetric stress element. Such a construction is not possible using $\mathbb P_k(T; \mathbb S)$ alone, but can be achieved by enriching it with $\mathbb B_k\Phi^f(\mathbb S)$.
By increasing the interior DoFs, we can construct a Raviart-Thomas (RT)-type element~\cite{RaviartThomas1977} with an enriched range.

We establish the inf-sup condition for $Q_{k-1,h}\div: \Sigma_{k,\phi}^{\operatorname{div}}(\mathcal{T}_h; \mathbb{S}) \to \mathbb P_{k-1}^{-1}(\mathcal T_h;\mathbb R^d)$ for all $k\geq 2$ without the constraint $k \geq d+1$. The space $\mathbb{B}_k \Phi^f(\mathbb{S})$ can be modified to $\mathbb{B}_k \Psi^f(\mathbb{S})$ so that it preserves the trace while changing the range: $\div \mathbb{B}_k \Psi^f(\mathbb{S}) \subset \mathbb P_{k-1}^{-1}(\mathcal T_h;\mathbb R^d)$. 

Finite element spaces on the barycentric refinement mesh $\mathcal T_h^{\rm R}$ can also be constructed, together with the corresponding inf-sup conditions.

The rest of this paper is organized as follows. Section~\ref{sec:preliminary} introduces simplices, barycentric refinement, and tangential-normal bases. The intrinsic construction of linear symmetric stress elements on the barycentric refinement is presented in Section~\ref{sec:linelement}. High-order elements on the barycentric refinement are developed in Section~\ref{sec:highorderelement}. Several discrete inf-sup conditions are established in Section~\ref{sec:infsup}. Finally, in Section~\ref{sec:discretization}, the symmetric stress elements on the barycentric refinement are applied to solve the linear elasticity equation.


\section{Preliminary}\label{sec:preliminary}
In this section, we present notation on simplexes and sub-simplexes, spaces, barycentric refinement, and the tangential-normal bases.

\subsection{Simplices, Complexes, and Triangulations}
For a $d$-dimensional simplex $T$, we let $\Delta(T)$ denote all the subsimplices of $T$, while $\Delta_{\ell}(T)$ denotes the set of subsimplices of dimension $\ell$, for $0\leq \ell \leq d$. Elements of $\Delta_0(T) = \{\texttt{v}_0, \ldots, \texttt{v}_d\}$ are $d+1$ vertices of $T$. 

To distinguish combinatorial and geometric structures, we introduce the abstract $d$-simplex $\TT$, a finite set with $d+1$ elements. The standard $d$-simplex is $\texttt{S}_d := \{0, 1, \ldots, d\}$. Any $\TT = \{\TT(0), \ldots, \TT(d)\}$ is combinatorially isomorphic to $\texttt{S}_d$ via $i \mapsto \TT(i)$.

A $d$-simplex $T$ with vertices $\texttt{v}_0, \ldots, \texttt{v}_d$ is a geometric realization of abstract simplex $\TT$ through $\TT(i) \mapsto \texttt{v}_i$, or of $\texttt{S}_d$ via $i \mapsto \texttt{v}_i$. The subset notation extends naturally: $\Delta_{\ell}(\TT)$ denotes the set of subsets of $\TT$ of cardinality $\ell+1$.

We use $f$ to denote both an abstract subset and its geometric realization. Algebraically, $f \in \Delta_{\ell}(\texttt{S}_d)$; geometrically,
$f$ is the $\ell$-simplex spanned by the corresponding vertices.
%
For $0\leq \ell \leq d-1$, the complement $f^* \in \Delta_{d-\ell -1} (\TS_d)$ satisfies $f \sqcup f^* = \{0, \ldots, d\}$ with disjoint union $\sqcup$. Geometrically, $f^*$ is the sub-simplex formed by the vertices not in $f$.

This notation simplifies indexing under the implicit embedding $i \mapsto \texttt{v}_i$. For example, $F_i := \{i\}^*$ denotes the $(d-1)$-dimensional face opposite to $\texttt{v}_i$, more concisely than $F_{\texttt{v}_i}$. Likewise, the tangential vector $\boldsymbol{t}_{i,j} := \texttt{v}_j - \texttt{v}_i$ is lighter than $\boldsymbol{t}_{\texttt{v}_i, \texttt{v}_j}$. A useful fact is that if $i,j\in f$, then $\bs t_{i,j}$ is tangent to $f$ and $\bs n_f\cdot \bs t_{i,j} = 0$, where $\bs n_f$ is a normal vector of $f$.

Let $\Omega \subset \mathbb{R}^d$ be a polyhedral domain with $d \geq 1$. A geometric triangulation $\mathcal{T}_h$ of $\Omega$ is a collection of $d$-simplices such that
\[
\bigcup_{T \in \mathcal{T}_h} T = \Omega, \qquad \mathring{T}_i \cap \mathring{T}_j = \emptyset \quad \text{for all } T_i \neq T_j \in \mathcal{T}_h,
\]
where $\mathring{T}$ denotes the interior of $T$. The subscript $h$ refers to the mesh size, i.e., the maximum diameter of all elements. We restrict to \emph{conforming triangulations}, where the intersection of any two simplices is either empty or a common subsimplex.

We adopt a topological viewpoint based on simplicial complexes to clarify the combinatorial structure~\cite{hatcher2005algebraic}.  
A \emph{simplicial complex} $\mathcal{S}$ over a finite vertex set $\mathcal V$ is a collection of subsets of $\mathcal V$ such that if $\TT \in \mathcal{S}$, then all subsets $\Delta(\TT)$ also belong to $\mathcal{S}$. Elements of $\mathcal V$ are \emph{vertices}, and elements of $\mathcal{S}$ are \emph{simplices}. Let $\Delta_\ell(\mathcal{S})$ be the set of all $\ell$-simplices in $\mathcal{S}$. A simplex $\TT$ is \emph{maximal} if it is not contained in any other simplex. The complex $\mathcal{S}$ is \emph{pure of dimension $d$} if all maximal simplices are $d$-simplices.
%

The geometric realization of the maximal simplices $\Delta_d(\mathcal{S})$ defines the triangulation $\mathcal{T}_h$. Following the finite element convention, we work with $\mathcal{T}_h$ and use $\Delta_\ell(\mathcal{T}_h)$ to denote the set of all $\ell$-simplices in the mesh.

\subsection{Spaces}
Set $\mathbb M:=\mathbb R^{d\times d}$.
Denote by $\mathbb{S}$ and $\mathbb{K}$ the subspaces of $\mathbb{M}$ consisting of symmetric and skew-symmetric matrices, respectively.
Any matrix $\boldsymbol{\tau}\in\mathbb M$ admits the decomposition 
$$\boldsymbol{\tau}=\sym\boldsymbol{\tau}+\skw\boldsymbol{\tau},$$ where the symmetric part $\sym\boldsymbol{\tau}=(\boldsymbol{\tau}+\boldsymbol{\tau}^{\intercal})/2$ and the skew-symmetric part $\skw\boldsymbol{\tau}=(\boldsymbol{\tau}-\boldsymbol{\tau}^{\intercal})/2$.
For a subspace $V\subset \mathbb R^d$, denote by $\mathbb S(V)$ and $\mathbb K(V)$ the spaces of the symmetric matrices and the skew-symmetric matrices restricted to $V$, respectively:
\begin{align*}
\mathbb S(V) &:= \sym (V\otimes V) = {\rm span}\{\sym(\boldsymbol{v}_i\otimes\boldsymbol{v}_j): 1\leq i\leq j\leq m\}, \\
\mathbb K(V) &:= \skw (V\otimes V) = {\rm span}\{\skw(\boldsymbol{v}_i\otimes\boldsymbol{v}_j): 1\leq i< j\leq m\},
\end{align*}
where $\{\boldsymbol{v}_1,\ldots, \boldsymbol{v}_m\}$ is a basis of $V$. 
For a space $B(D)$ defined on $D$,
let $B(D; \mathbb{X}):=B(D)\otimes\mathbb{X}$ be its vector or tensor version for $\mathbb{X}$ being $\mathbb{R}^d$, $\mathbb{S}$ and $\mathbb{K}$.

 
Denote by $\mathbb P_k(\Omega)$ the space of polynomials of degree $k$ on a domain $\Omega$. Denote by $\mathbb P_k^{-1}(\mathcal T_h) = \{ v\in L^2(\Omega): v\mid_T\in \mathbb P_k(T), \forall\,T\in \mathcal T_h\}$ the discontinuous polynomial space of degree $k$ on $\mathcal T_h$. Let $Q_{k,\Omega}: L^2(\Omega) \to \mathbb{P}_k(\Omega)$ and $Q_{k,h}: L^2(\Omega) \to \mathbb{P}_k^{-1}(\mathcal{T}_h)$ denote the $L^2$-projection operators, extended in the natural way to vector- and tensor-valued functions.

Introduce the rigid motion space on simplex $T$~\cite{Nedelec1980}
$$
\textrm{RM}(T)=\mathbb{P}_0(T; \mathbb{R}^d) + \mathbb{P}_0(T; \mathbb{K})\boldsymbol{x},
$$
where $\boldsymbol{x}$ is the position vector on $T$. Let $Q_{\rm RM}: L^2(T; \mathbb{R}^d)\to \textrm{RM}(T)$ be the $L^2$ projection operator.

The space $H(\div, \Omega; \mathbb S):=\{\bs \tau\in L^2(\Omega;\mathbb S): \div \bs \tau\in L^2(\Omega; \mathbb R^d)\}.$ For a subdomain $D \subseteq \Omega$, the trace operator for the div operator is
$$
{\rm tr}_D^{\div} \boldsymbol \tau = \boldsymbol \tau \boldsymbol n|_{\partial D} \quad \textrm{ for }\;\;\boldsymbol{\tau}\in C(D;\mathbb S),
$$
where $\bs n$ denotes the outwards unit normal vector of $\partial D$, and $C(D; \mathbb{S}) := C(D) \otimes \mathbb{S}$, with $C(D)$ denoting the space of continuous functions on $D$. The trace operator ${\rm tr}_D^{\div}$ can be continuously extended to ${\rm tr}_D^{\div}: H(\div, D; \mathbb S) \to H^{-1/2}(\partial D; \mathbb R^d)$.

We define the space
$$
H_0(\div, D; \mathbb S) := H(\div, D; \mathbb S)\cap \ker({\rm tr}_D^{\div})=\{\bs \tau\in H(\div, D; \mathbb S): {\rm tr}_D^{\div} \boldsymbol \tau=0\}.
$$ 
Given a conforming triangulation $\mathcal T_h$ of $\Omega$ and a piecewise smooth function $\bs \tau$, it is well known that $\bs \tau\in H(\div, \Omega; \mathbb S)$ if and only if $\bs \tau \bs n_F$ is continuous across all faces $F\in \Delta_{d-1}(\mathcal T_h)$, where $\bs n_F$ is a fixed unit normal vector of $F$. 

Given a $(d-1)$-dimensional face $F$ and a vector $\boldsymbol v\in \mathbb R^d$, define 
$$
\Pi_F\boldsymbol v:= 
 (\boldsymbol I - \boldsymbol n_F\boldsymbol n_F^{\intercal})\boldsymbol v
$$ 
as the projection of $\boldsymbol v$ onto the face $F$.

\subsection{Barycentric refinement}
Algebraically a barycentric refinement of an abstract simplex $\texttt{T}$ is obtained by adding one more vertex, indexed by $\TT(c)$. Let $\TT_i = \TT\setminus \{\TT(i)\} \cup \{\TT(c)\}$ be the abstract simplex by replacing $\TT(i)$ by $\TT(c)$. The abstract split $\TT^{\rm R} := \{ \TT_i \mid i=0, 1, \ldots, d\}$.

The index set is better described by an abstract simplex and its split. We extend the complement notation. For $f\subseteq \{0, 1, \ldots, d, c\}$, define $f^c$ s.t.
$$
f \sqcup f^c =  \{0, 1, \ldots, d, c\}.
$$
When $f\in \Delta_{\ell}(\TT)$, $f^c = f^* \cup \{c\}.$ In this notation, $\TT_i = \{i\}^{c}= F_i\cup \{c\}$  for $i=0, 1, \ldots, d$. 

Geometrically, the barycentric refinement $T^{\rm R}$ of $T$ is obtained by mapping $\TT(c)$ to $\texttt{v}_{c}$, the barycenter of a $d$-dimensional simplex $T$ with vertices $\texttt{v}_0, \texttt{v}_1, \ldots, \texttt{v}_d$, i.e., $\texttt{v}_c = (d+1)^{-1}\sum_{i=0}^d \texttt{v}_i$. The corresponding geometric embedding of $\TT_i$ will be denoted by $T_i$.
For $i = 0, 1, \ldots, d$, we have 
$$
\boldsymbol{t}_{i,c}=\texttt{v}_c-\texttt{v}_i=\frac{1}{d+1}\sum_{j=0}^d(\texttt{v}_{j}-\texttt{v}_{i})=\frac{1}{d+1}\sum_{j=0}^d\boldsymbol{t}_{i,j},
$$
and thus
\begin{equation*}
\sum_{i=0}^d \bs t_{i,c} = \frac{1}{d+1}\sum_{i,j=0}^d\boldsymbol{t}_{i,j} = 0.
\end{equation*}



We will refine our notation on $(d-1)$-dimensional faces by including a simplex. For $i=0,\ldots, d$, denote by $F_i(T)$ the $(d-1)$-dimensional face of $T$ opposite to $\texttt{v}_i$, and  by $\boldsymbol{n}_{F_i}$ its unit normal vector outward to $T$. Algebraically, $F_i(\TT) = \{i\}^* =  \texttt{S}_d\backslash \{ i\}$. By changing the simplex to $T_0$, $F_i(\TT_0) = \{0\}^c\cap \{i\}^c = \{0, i\}^c$ whose geometric realization is the face of $T_0$ opposite to $\texttt{v}_i$ for $i\in \{0\}^*=\{1,\ldots, d\}$.

For $f\in \Delta_{\ell}(\TT)$ with $0\leq \ell\leq d-1$, and $i\in f^*$, we have $F_i(\TT) = \{i\}^* \supseteq (f^*)^* = f$, i.e., $f\in \Delta_{\ell}(F_i(\TT))$ for $i\in f^*$, and similarly $\TT_i = \{i\}^c \supseteq (f^c)^c =  f$, i.e. $f \in \Delta_{\ell}(\TT_i)$ for $i\in f^*$. For $i,j = 0,\ldots, d$, the intersection $F_{ij}:= T_i\cap T_j$ is a $(d-1)$-dimensional face containing $\texttt{v}_c$ but not $\texttt{v}_{i}, \texttt{v}_{j}$. Algebraically, $F_{ij} = \{i\}^c\cap \{j\}^c  = \{i,j\}^c$. Treating a sub-simplex as a subset clarifies the geometry through algebraic operations, in the spirit of Descartes.

\begin{lemma}
Let $F_{ij}:= T_i\cap T_j$ be the $(d-1)$-dimensional face containing $\texttt{v}_c$ but not $\texttt{v}_{i}, \texttt{v}_{j}$, for $0\leq i< j \leq d$, and $\bs n_{F_{ij}}$ be a normal vector of $F$. Then
 \begin{equation}\label{eq:ti+tj}
(\boldsymbol{t}_{i,c}+\boldsymbol{t}_{j,c})\cdot\boldsymbol{n}_{F_{ij}} = 0.
\end{equation}
\end{lemma}
\begin{proof}
Algebraically $F_{ij} = \{i\}^c\cap \{j\}^c  = \{i,j\}^c$. As $c\in F_{ij}$ and $\ell \in F_{ij}$ for $\ell\in \{i,j\}^*$, we have
\begin{equation}\label{eq:nFtlc}
\boldsymbol{n}_{F_{ij}} \cdot \boldsymbol{t}_{\ell,c} = 0, \quad \ell \in \{i,j\}^*.
\end{equation}

By $\boldsymbol{t}_{0, c}+\boldsymbol{t}_{1, c}+\ldots+\boldsymbol{t}_{d, c}=0$ and \eqref{eq:nFtlc}, it follows 
$$
(\boldsymbol{t}_{i,c}+\boldsymbol{t}_{j,c})\cdot\boldsymbol{n}_{F_{ij}} = -\sum_{\ell\in \{ i,j\}^*}\boldsymbol{t}_{\ell,c}\cdot\boldsymbol{n}_{F_{ij}} = 0.
$$
\end{proof}

Denote by $\Delta_{\ell}(\mathring{T}^{\rm R}) = \Delta_{\ell}(T^{\rm R}) \setminus \Delta_{\ell}(T)$ the set of all $\ell$-dimensional subsimplices inside $T^{\rm R}$ that contain the barycenter $\texttt{v}_c$. For a conforming mesh $\mathcal T_h$, let $\mathcal T_h^{\rm R}$ be the barycentric refinement of $\mathcal T_h$. Denote by $\Delta_{\ell}(\mathring{\mathcal{T}}_h^{\rm R}) = \Delta_{\ell}(\mathcal{T}_h^{\rm R})\backslash \Delta_{\ell}(\mathcal{T}_h)$. 

We use $\lambda_i$ to denote the barycentric coordinate of $T$ corresponding to $\texttt{v}_i$. Then $\nabla \lambda_i = - h_i^{-1} \bs n_{F_i}$, where $h_i$ is the distance of $\texttt{v}_i$ to the face $F_i(T)$. For $f\in \Delta_{\ell}(T)$, the bubble polynomial $b_f := \prod_{i\in f}\lambda_i\in \mathbb P_{\ell+1}(f)$ and can be extended to $\mathbb P_{\ell+1}(T)$ using the barycentric coordinate. 

Introduce the linear Lagrange space 
$$
V^{L}_1(T^{\rm R}) := \mathbb{P}_1^{-1}(T^{\rm R}) \cap H^1(T)
= \{ v \in C(T): v|_{T_i} \in \mathbb{P}_1(T_i),\; T_i \in T^{\rm R},\; i = 0, \ldots, d \}.
$$
For the refined element $T^{\rm R}$, let $\lambda_i^{\rm R} \in V^{L}_1(T^{\rm R})$ denote the piecewise linear function such that 
$\lambda_i^{\rm R}(\texttt{v}_j) = \delta_{i,j}$ for vertices $\texttt{v}_i, \texttt{v}_j \in \Delta_0(T^{\rm R})$, where \(\delta_{i,j}\) denotes the Kronecker delta for \(i, j = 0, 1, \ldots, d, c\). 
On each subelement $T_i$ of $T^{\rm R}$ $(i = 0, 1, \ldots, d)$, it agrees with the barycentric coordinate on $T_i$.

For $F\in \Delta_{d-1}(\mathring{\mathcal{T}}_h^{\rm R})$, let $T_1, T_2 \in \mathcal T_h^{\rm R}$ so that $F=\partial T_1 \cap \partial T_2$ and the fixed normal vector $\boldsymbol{n}_F$ coincides with the outward unit normal to $\partial T_1$.
For piecewise smooth function $v$ defined on $\Omega$, the jump of $v$ on face $F$ is defined by
\begin{equation*}
[v]|_F= (v|_{T_1})|_{F} - (v|_{T_2})|_{F}.
\end{equation*}

\subsection{Tangential-normal ($t$-$n$) bases}  
For a subsimplex \( f \in \Delta_{\ell}(T) \), let us select \(\ell\) linearly independent tangential vectors \(\{\boldsymbol{t}_1^f, \ldots, \boldsymbol{t}_{\ell}^f\}\) along \(f\) and \(d - \ell\) linearly independent normal vectors \(\{\boldsymbol{n}_1^f, \ldots, \boldsymbol{n}_{d-\ell}^f\}\) orthogonal to \(f\). While the vectors can be normalized, the sets \(\{\boldsymbol{t}_i^f\}\) and \(\{\boldsymbol{n}_i^f\}\) are not necessarily orthonormal. Together, these \(d\) vectors \(\{\boldsymbol{t}_1^f, \ldots, \boldsymbol{t}_{\ell}^f, \boldsymbol{n}_1^f, \ldots, \boldsymbol{n}_{d-\ell}^f\}\) form a basis for \(\mathbb{R}^d\).  

The tangent space and normal space of \(f\) are defined, respectively, as follows:  
\[
\mathscr{T}^f := \operatorname{span} \{\boldsymbol{t}_1^f, \ldots, \boldsymbol{t}_{\ell}^f\},  
\quad  
\mathscr{N}^f := \operatorname{span} \{\boldsymbol{n}_1^f, \ldots, \boldsymbol{n}_{d-\ell}^f\}.
\]  
These subspaces satisfy \(\mathbb{R}^d = \mathscr{T}^f \oplus \mathscr{N}^f\). If the normal basis \(\{\boldsymbol{n}_j^f\}\) is determined solely by \(f\) and does not vary with either the \((d-1)\)-dimensional face \(F\) or the element \(T\) containing \(f\), it is referred to as a global normal basis.  

Within \(\mathscr{T}^f\), we can define a dual basis \(\{\hat{\boldsymbol{t}}_1^f, \ldots, \hat{\boldsymbol{t}}_{\ell}^f\}\) such that \(\hat{\boldsymbol{t}}_i^f \in \mathscr{T}^f\) and \(\hat{\boldsymbol{t}}_i^f \cdot \boldsymbol{t}_j^f = \delta_{i,j}\). Similarly, a dual basis \(\{\hat{\boldsymbol{n}}_1^f, \ldots, \hat{\boldsymbol{n}}_{d-\ell}^f\}\) can be identified within \(\mathscr{N}^f\) such that \((\hat{\boldsymbol{n}}_i^f, \boldsymbol{n}_j^f) = \delta_{i,j}\) for \(i, j = 1, \ldots, d-\ell\). Since \(\mathscr{T}^f \perp \mathscr{N}^f\), the combined set \(\{\hat{\boldsymbol{t}}_1^f, \ldots, \hat{\boldsymbol{t}}_{\ell}^f, \hat{\boldsymbol{n}}_1^f, \ldots, \hat{\boldsymbol{n}}_{d-\ell}^f\}\) serves as the dual basis of \(\{\boldsymbol{t}_1^f, \ldots, \boldsymbol{t}_{\ell}^f, \boldsymbol{n}_1^f, \ldots, \boldsymbol{n}_{d-\ell}^f\}\). When $\delta_{i,j}$ is replaced by $\delta_{i,j}c_{i}$ with $c_i\neq 0$, those two bases are called biorthogonal or scaled dual bases. 

Focusing on the subspace \(\mathscr{N}^f\), two distinct bases are useful. A basis for \(\mathscr{N}^f\) can be formed using the unit normal vectors associated with faces $F_i$ containing $f$:  
\[
\mathscr{N}^f = {\rm span} \{\boldsymbol{n}_{F_i} \mid i \in f^*\},
\]  
which we term the face normal basis.

For \(f \in \Delta_{\ell}(T)\) with \(\ell = 0, 1, \ldots, d-1\), and \(i \in f^*\), let \(f \cup \{i\}\) denote the \((\ell+1)\)-dimensional face containing \(f\) and vertex \(\texttt{v}_i\). Let \(\boldsymbol{n}_{f \cup \{i\}}^f\) be the unit vector normal to \(f\) but tangential to \(f \cup \{i\}\), inheriting its orientation. Then,  
\[
\mathscr{N}^f = {\rm span} \{\boldsymbol{n}_{f \cup \{i\}}^f \mid i \in f^*\}
\]  
forms a basis for \(\mathscr{N}^f\) which is called the tangential-normal basis.  

The face normal $\bs n_{F_i}$ is a normlization of $\nabla \lambda_i$ and $\boldsymbol{n}_{f \cup \{i\}}^f$ is a normalization of $\nabla_{f\cup \{i\}}\lambda_i$, where $\nabla_{f\cup \{i\}}$ is the surface gradient.

The following result, detailed in~\cite{ChenHuang2024}, establishes the relationship between these bases. Illustrations for the three-dimensional case can be found in~\cite[Fig. 1]{ChenHuang2024} and~\cite{ChenHuang2024Tangential-Normal}.  

\begin{lemma}
 For $f \in \Delta_{\ell}(T)$ with $\ell = 0, 1, \ldots, d-1$, the scaled tangential-normal basis  
\[
\left\{ \frac{\boldsymbol{n}_{f \cup \{i\}}^f}{\boldsymbol{n}_{f \cup \{i\}}^f \cdot \boldsymbol{n}_{F_i}} \mid i \in f^* \right\}  
\]  
is dual to the face normal basis \(\{\boldsymbol{n}_{F_i} \mid i \in f^*\}\) in \(\mathscr{N}^f\).  
\end{lemma}
\begin{proof}
For $i\neq j$ and $i,j\in f^*$, we conclude from $\{j\}\subseteq f^*$ and $\{j\}\subseteq \{i\}^*$ that $\{j\}\subseteq (f^*\cap \{i\}^*)$. Taking the complement to get $(f\cup \{i\})\subseteq F_j$. So the normal vector $\bs n_{F_j}$ is orthogonal to $\mathscr T^{f\cup \{i\}}$ which contains $\boldsymbol{n}_{f \cup \{i\}}^f$. That is $ \bs n_{F_j} \cdot \boldsymbol{n}_{f \cup \{i\}}^f=0$ for  $i\neq j$ and $i,j\in f^*$.
\end{proof}

An important example is taking $f$ as a vertex. Without loss of generality, take $f = \{0\}$. Then the non-normalized tangential-normal vector is $\bs t_{0,i}$ and the non-normalized face normal vector is $\nabla \lambda_i$. The duality reads as
\begin{equation}\label{eq:duality0}
\nabla \lambda_i \cdot \bs t_{0,j} = \delta_{i,j}, \quad 1\leq i, j\leq d,
\end{equation}
which can be easily verified by evaluating the constant $\nabla \lambda_i \cdot \bs t_{0,j}$ at ending vertices.

\subsection{$t$-$n$ decomposition of symmetric tensors}\label{sec:tnS}
For $f\in\Delta_{\ell}(T)$ with $\ell=0, \ldots, d$, we choose a \(t\)-\(n\) basis \(\{\boldsymbol{t}_i^f, \boldsymbol{n}_j^f\}_{i=1,\ldots,\ell}^{j=1,\ldots,d-\ell}\). 
%
It is straightforward to verify the direct decomposition:  
\begin{equation}\label{Sdecomp}
\mathbb{S} = \underbrace{\mathbb S( \mathscr T^f )}_{\mathscr{T}^f(\mathbb{S})} \oplus \underbrace{\mathbb S( \mathscr N^f )\oplus \sym ( \mathscr T^f\otimes \mathscr N^f)}_{\mathscr{N}^f(\mathbb{S})},
\end{equation}
where in view of $t$-$n$ basis
\begin{align*}
\mathscr T^f(\mathbb S) &= \mathbb S( \mathscr T^f ) = \textrm{span}\big\{\sym(\bs t_i^f\otimes \bs t_j^f),  1\leq i\leq j\leq \ell\big\},\\
\mathscr N^f(\mathbb S) &= \textrm{span}\big\{\sym(\bs n_i^f\otimes \bs n_j^f),  1\leq i\leq j\leq d-\ell\big\} \\
&\quad\;\oplus\textrm{span}\big\{\sym(\bs t_i^f\otimes \bs n_j^f),  1\leq i\leq\ell,1\leq j\leq d-\ell\big\}. 
\end{align*}

We refer to~\cite[Fig. 4]{ChenHuang2024} for graphical illustrations of this decomposition for \(f\) being vertices, edges, and faces in three dimensions. Note that there is no \(\mathscr{T}^f(\mathbb{S})\) for \(\dim f = 0\) and no \(\mathscr{N}^f(\mathbb{S})\) for \(\dim f = d\).

\section{Linear Elements}\label{sec:linelement}

We present the intrinsic characterization and construction of linear symmetric stress elements on the barycentric refinement.

\subsection{Piecewise constant element}

With the barycentric split, the space $\mathbb{P}_0(T;\mathbb{S})$ can be enlarged to the piecewise constant symmetric tensor space $\mathbb{P}_0^{-1}(T^{\rm R}; \mathbb{S})$. However, due to the normal continuity required for membership in $H(\div, T; \mathbb{S})$, no additional functions are admissible.

We first perform a dimension count: $\dim \mathbb{P}_0^{-1}(T^{\rm R}; \mathbb{S}) = (d+1)\times \frac{1}{2}d(d+1)$. There are $\frac{1}{2}d(d+1)$ interior $(d-1)$-dimensional faces $F_{ij}, 0\leq i < j\leq d$, each imposing $d$ constraints to enforce the continuity of $\boldsymbol{\sigma} \boldsymbol{n}$. Subtracting the number of constraints gives
\[
\dim \mathbb{P}_0^{-1}(T^{\rm R}; \mathbb{S}) - \# \text{constraints} = \frac{1}{2}d(d+1) = \dim \mathbb{P}_0(T;\mathbb{S}).
\]
This dimension count is not a rigorous proof, as the constraints must also be shown to be linearly independent.

A sketch of a rigorous proof is as follows: assuming $\boldsymbol{\tau} \boldsymbol{n} |_F = 0$, we can expand $\boldsymbol{\tau}$ in $\mathbb{S}(\mathscr{T}^F)$. Then, considering two intersecting faces, we apply the normal continuity and symmetry of $\boldsymbol{\tau}$ and use properties of the barycenter to show that all expansion coefficients vanish.


\begin{lemma}\label{lem:Stensor1}
We have $H(\div,T; \mathbb S)\cap \mathbb P_{0}^{-1}(T^{\rm R}; \mathbb S)=\mathbb P_0(T;\mathbb S)$.
\end{lemma}

\begin{proof}
It is evident that $\mathbb P_0(T;\mathbb S)\subseteq (H(\div, T; \mathbb S)\cap \mathbb P_{0}^{-1}(T^{\rm R}; \mathbb S))$. We now prove the reverse inclusion.

Take $\boldsymbol{\sigma}\in (H(\div,T; \mathbb S)\cap \mathbb P_{0}^{-1}(T^{\rm R}; \mathbb S))$. Define $\boldsymbol{\tau}\in (H(\div,T; \mathbb S)\cap \mathbb P_{0}^{-1}(T^{\rm R}; \mathbb S))$ by setting $\boldsymbol{\tau}|_{T_i}=\boldsymbol{\tau}_i := \boldsymbol{\sigma}|_{T_i}-\boldsymbol{\sigma}|_{T_0}$ for $i=0,1,\ldots,d$. Clearly, $\boldsymbol{\tau}_0 = 0$, and $(\boldsymbol{\tau}\boldsymbol{n})|_{F_i(T_0)} = 0$ for $i=1, 2, \ldots, d$. We shall prove $\boldsymbol{\tau} = 0$ and consequently $\boldsymbol{\sigma}|_{T_i} = \boldsymbol{\sigma}|_{T_0}$, i.e., $\boldsymbol{\sigma} \in \mathbb P_0(T;\mathbb S)$.

\begin{figure}[htbp]
\subfigure[3D]{
\begin{minipage}[t]{0.5\linewidth}
\centering
\includegraphics*[width=4.25cm]{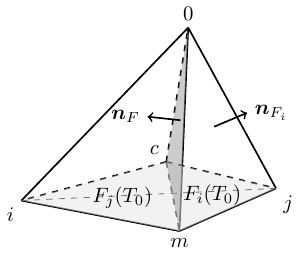}
\end{minipage}}
\subfigure[2D]
{\begin{minipage}[t]{0.5\linewidth}
\centering
\includegraphics*[width=4.125cm]{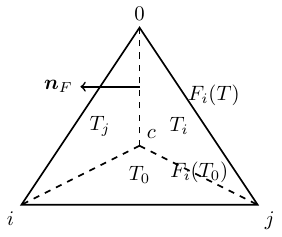}
\end{minipage}}
\caption{Barycentric refinement.}
\label{fig:refinement}
\end{figure}

We illustrate the idea using a 2D example before generalizing to higher dimensions; see Fig.~\ref{fig:refinement}~(b). Inside $T_i$, we use the $t$-$n$ basis $\{\bs t\otimes \bs t, \sym(\bs t\otimes \bs n), \bs n\otimes \bs n\}$ of the face $F_i(T_0)$ to expand $\mathbb S$. Since $\boldsymbol{\tau}\boldsymbol{n}|_{F_i(T_0)} = 0$, we deduce that only the tangential-tangential component remains, i.e., $\boldsymbol{\tau}_i = \tau_{i,j,j}\boldsymbol{t}_{j,c}\otimes \boldsymbol{t}_{j,c}$ and $\boldsymbol{\tau}_j = \tau_{j,i,i}\boldsymbol{t}_{i,c}\otimes \boldsymbol{t}_{i,c}$. 

Let $F = F_{ij} = \{i, j\}^c$ be the $(d-1)$-dimensional face without vertices $\texttt{v}_i$ and $\texttt{v}_j$. Its normal vector is denoted by $\boldsymbol{n}_F$. The continuity condition $(\boldsymbol{\tau}_i\boldsymbol{n}_F)(\texttt{v}_c) = (\boldsymbol{\tau}_j\boldsymbol{n}_F)(\texttt{v}_c)$ implies
\[
\tau_{i,j,j} \boldsymbol{t}_{j,c} (\boldsymbol{t}_{j,c}\cdot \boldsymbol{n}_F) = \tau_{j,i,i} \boldsymbol{t}_{i,c} (\boldsymbol{t}_{i,c}\cdot \boldsymbol{n}_F).
\]
As $\boldsymbol{t}_{j,c}$ and $\boldsymbol{t}_{i,c}$ are linearly independent and $(\boldsymbol{t}_{j,c}\cdot \boldsymbol{n}_F)(\boldsymbol{t}_{i,c}\cdot \boldsymbol{n}_F) \neq 0$, we conclude $\tau_{i,j,j} = \tau_{j,i,i} = 0$, i.e., $\boldsymbol{\tau}_i = \boldsymbol{\tau}_j = 0$.

Now consider the general case in $d$ dimensions. 
We choose $\{\boldsymbol{t}_{m,c} \mid m \in \{0, i\}^*\}$ as a basis of the tangential plane of $F_i(T_0)$.
As $(\boldsymbol{\tau}\boldsymbol{n})|_{F_i(T_0)}=0$, we have $\boldsymbol{\tau}_i|_{F_i(T_0)}\in \mathbb S(\mathscr T^{F_i(T_0)})$ and can express $\boldsymbol{\tau}_i$ as
\begin{equation*}
\boldsymbol{\tau}_i =  \sum_{m,n\in \{0, i\}^*}{\tau}_{i, m, n}\boldsymbol{t}_{m,c}\otimes\boldsymbol{t}_{n,c}.
\end{equation*}
Clearly ${\tau}_{i, m, n}={\tau}_{i, n, m}$, for $m,n\in \{0, i\}^*$, as $\bs \tau_i$ is symmetric. We will use the normal continuity to conclude all coefficients ${\tau}_{i, m, n}$ vanish.

Let $F = F_{ij} = \{i, j\}^c$ be the $(d-1)$-dimensional face shared by $T_i$ and $T_j$. Fix a unit normal vector $\boldsymbol{n}_F$. 
Evaluating $\bs \tau\bs n_F$ at $\texttt{v}_c$ and using \eqref{eq:nFtlc}, we find
\[
(\boldsymbol{\tau}_i\boldsymbol{n}_F)(\texttt{v}_c) = \sum_{m \in \{ 0, i\}^*} \tau_{i,m,j}\boldsymbol{t}_{m,c} (\boldsymbol{t}_{j,c}\cdot \boldsymbol{n}_F),
\]
and similarly,
\[
(\boldsymbol{\tau}_j\boldsymbol{n}_F)(\texttt{v}_c) = \sum_{m \in \{ 0, j\}^*} \tau_{j,m,i}\boldsymbol{t}_{m,c}(\boldsymbol{t}_{i,c}\cdot \boldsymbol{n}_F).
\]
Expanding the identity $(\boldsymbol{\tau}_i\boldsymbol{n}_F)(\texttt{v}_c) = (\boldsymbol{\tau}_j\boldsymbol{n}_F)(\texttt{v}_c)$ in the basis $\{\boldsymbol{t}_{m,c}, m=1, \ldots, d\}$, we conclude that all coefficients vanish as follows.

Like the 2D case, we have all diagonal entries ${\tau}_{i, j, j}={\tau}_{j, i,i}=0$ as $(\boldsymbol{t}_{j,c}\cdot \boldsymbol{n}_F)(\boldsymbol{t}_{i,c}\cdot \boldsymbol{n}_F) \neq 0$. Moreover, as the coefficient of $\bs t_{m,c}$, 
$$
{\tau}_{i, m,j}(\boldsymbol{t}_{j,c}\cdot\boldsymbol{n}_F)
={\tau}_{j, m,i}(\boldsymbol{t}_{i,c}\cdot\boldsymbol{n}_F),\quad i\neq j, i\neq m, j\neq m.
$$
Using the relation \eqref{eq:ti+tj}, we get
$$
(\boldsymbol{t}_{j,c}\cdot\boldsymbol{n}_F)({\tau}_{i, m,j}+{\tau}_{j, m,i})=0,\quad i\neq j, i\neq m, j\neq m.
$$
By $\boldsymbol{t}_{j,c}\cdot\boldsymbol{n}_F\neq0$ and the symmetry of the last two indices in ${\tau}_{i, m,j}$ and ${\tau}_{j, m,i}$, we acquire the skew-symmetry of the first two indices
$$
{\tau}_{i, j,m}+{\tau}_{j, i,m}=0\quad\textrm{ for } i,j,m\neq0, i\neq j, i\neq m, j\neq m.
$$
Fix pairwise distinct $i,j,m\neq 0$ and set
\[
a:=\tau_{i,j,m},\qquad b:=\tau_{j,m,i},\qquad c:=\tau_{m,i,j}.
\]
Using skew-symmetry in the first two indices and symmetry in the last two, we obtain the linear relations
\[
a+b = \tau_{i,j,m} + \tau_{j,m,i}
= \tau_{i,j,m} + \tau_{j,i,m} = 0,
\]
\[
b+c = \tau_{j,m,i} + \tau_{m,i,j}
= \tau_{j,m,i} + \tau_{m,j,i} = 0,
\]
\[
a+c = \tau_{i,j,m} + \tau_{m,i,j}
= \tau_{i,m,j} + \tau_{m,i,j} = 0,
\]
where each equality uses symmetry to swap the last two indices and skew-symmetry to swap the first two.
%
Thus $a = b = c = 0$, i.e.,
\[
\tau_{i,j,m}=\tau_{j,m,i}=\tau_{m,i,j}=0
\]
for all pairwise distinct $i,j,m\neq 0$, as claimed.

%
Hence,
$\boldsymbol{\tau}_i(\texttt{v}_c)=0$ for $i=1,\ldots, d$. As a piecewise constant function, this means $\boldsymbol{\tau}|_{T_i}=0$ and $\boldsymbol{\sigma}\in\mathbb P_0(T;\mathbb S)$.
\end{proof}

At vertex \(\texttt{v}_0\), define the piecewise constant function as follows:  
\begin{equation}\label{eq:phi}
\boldsymbol{\phi}_{ij}^{0} := \chi_{T_{i}}\sym(\boldsymbol{t}_{0, c} \otimes \boldsymbol{t}_{0, j}) - \chi_{T_{j}}\sym(\boldsymbol{t}_{0, c} \otimes \boldsymbol{t}_{0, i}), \quad 1\leq i, j\leq d.
\end{equation}
Notice that $\boldsymbol{\phi}^{0}_{ii} = 0$ and $\boldsymbol{\phi}^{0}_{ij} = - \boldsymbol{\phi}^{0}_{ji}$.
The associated space is given by  
\[
\Phi^{0}(\mathbb{S}) := \text{span} \left\{ \boldsymbol{\phi}_{ij}^{0} \mid 1\leq i<j\leq d \right\}.
\]
In the following proof, we still recommend using Fig. \ref{fig:refinement} for tracking the index. 

\begin{lemma}\label{lem:phiijv}
For $1\leq i,j\leq d$, $i\neq j$, 
we have $\supp(\boldsymbol{\phi}_{ij}^{0}) = T_i\cup T_j$, and $\boldsymbol{\phi}_{ij}^{0}|_{T\setminus T_0}\in H(\div,T\setminus T_0; \mathbb S)$.
\end{lemma}
\begin{proof}
First, $\boldsymbol{\phi}_{ij}^{0}\big|_{T\setminus (T_i\cup T_j)} = 0$ follows from the definition of the characteristic functions $\chi_{T_{i}}$ and $\chi_{T_{j}}$. Next, we show that $\boldsymbol{\phi}_{ij}^{0}|_{T\setminus T_0}\in H(\div,T\setminus T_0; \mathbb S)$.
It is equivalent to proving that 
\begin{equation}\label{eq:phivijdivconformity0}
[\boldsymbol{\phi}_{ij}^{0} \boldsymbol{n}]|_F = 0, \quad \forall~ F \in \Delta_{d-1}(\mathring{T}^{\rm R}) \setminus \Delta_{d-1}(T_0).
\end{equation}

For face $F \in \Delta_{d-1}(\mathring{T}^{\rm R})\setminus \Delta_{d-1}(T_0)$, clearly $\{0, c\} \in F$ and consequently $\boldsymbol{t}_{0, c} \cdot \boldsymbol{n}_F = 0$. We then verify \eqref{eq:phivijdivconformity0} by considering the following cases.

\smallskip

\step{1}  \(i, j\in F\).  Then \(\boldsymbol{t}_{0, c} \cdot \boldsymbol{n}_F = \boldsymbol{t}_{0, i} \cdot \boldsymbol{n}_F = \boldsymbol{t}_{0, j} \cdot \boldsymbol{n}_F = 0\), \eqref{eq:phivijdivconformity0} follows from the definition of $\phi^0_{ij}$ given in~\eqref{eq:phi}.

\smallskip

\step{2} \( j\in F\) and \(i \notin F\).  
As $\TT_j = \{ j\}^c$, $j\not\in \TT_j$, and $j\in F$ implies $F\not\in \Delta_{d-1}(\TT_j)$. We write $i \notin F$ as $i\in F^c$ which implies $F\subset \{i\}^c = \TT_i$, i.e. $F\in \Delta_{d-1}(\TT_i)$. Therefore two simplices containing $F$ are $\TT_i$ and $\TT_{\ell}$ for some $\ell \neq i, j$.
From \(\boldsymbol{t}_{0, c} \cdot \boldsymbol{n}_F = \boldsymbol{t}_{0, j} \cdot \boldsymbol{n}_F = 0\), we conclude 
\[
\boldsymbol{\phi}_{ij}^{0}|_{T_i} \boldsymbol{n}_F = \frac{1}{2} \boldsymbol{t}_{0, c} (\boldsymbol{t}_{0, j} \cdot \boldsymbol{n}_F) = 0 = \boldsymbol{\phi}_{ij}^{0}|_{T_{\ell}} \boldsymbol{n}_F \quad \text{on } F.
\]
Thus, \eqref{eq:phivijdivconformity0} holds. A similar argument applies when $i\in F$ but $j\not\in F$.

\smallskip

\step{3} \( i,j \notin F \). Then $F = F_{ij} = \{i, j\}^c$. 
Using the relation \(\boldsymbol{t}_{0, i} = \boldsymbol{t}_{0, c} + \boldsymbol{t}_{c, i}\) and the identity \(\boldsymbol{t}_{c, 0} + \boldsymbol{t}_{c, 1} + \cdots + \boldsymbol{t}_{c, d} = 0\), it follows from \eqref{eq:ti+tj} that
\[
\frac{1}{2}(\boldsymbol{t}_{0, i} + \boldsymbol{t}_{0, j}) \cdot \boldsymbol{n}_{F} = \frac{1}{2}(\boldsymbol{t}_{c, i} + \boldsymbol{t}_{c, j}) \cdot \boldsymbol{n}_{F} = 0.
\]
Hence,  
\[
\boldsymbol{\phi}_{ij}^{0}|_{T_i} \boldsymbol{n}_{F} - \boldsymbol{\phi}_{ij}^{0}|_{T_j} \boldsymbol{n}_{F} = \frac{1}{2} \boldsymbol{t}_{0, c} (\boldsymbol{t}_{0, j} \cdot \boldsymbol{n}_{F} + \boldsymbol{t}_{0, i} \cdot \boldsymbol{n}_{F}) = 0 \quad \text{on } F.
\]

Thus, \eqref{eq:phivijdivconformity0} holds in all cases.  
\end{proof}



From $\boldsymbol{\phi}_{ij}^{0}\big|_{T\setminus (T_i\cup T_j)} = 0$, we observe that \( \boldsymbol{\phi}_{ij}^{0}\bs n_F|_F \neq 0 \) on the outer faces \(F_i(T) \cup F_j(T)\) and the interior faces $F_{\ell}(T_i) \cup F_{\ell}(T_j)$ for $\ell\in\{i,j\}^*$. At vertex $\texttt{v}_0$, the dimension of \(\mathbb{S}\) is \(d(d+1)/2\). The trace \((\boldsymbol{\sigma} \boldsymbol{n})|_F \in \mathbb{R}^d\) contributes \(d \times d\) degrees of freedom on the \(d\) faces \(F \in \Delta_{d-1}(T)\) containing \(\texttt{v}_0\). By introducing one \(\boldsymbol{\phi}_{ij}^{0}\) for each pair \(i < j, i, j=1,\ldots, d\), we achieve a sufficient number of basis functions to match the facewise degrees of freedom.  

\begin{lemma}\label{lem:vertexDoF}
Let $\boldsymbol{\tau}\in H(\div,T\setminus T_0; \mathbb S)\cap \mathbb P_{0}^{-1}(T^{\rm R}\setminus T_0; \mathbb S)$.
If
\[
\boldsymbol{\tau} \boldsymbol{n}_F(\texttt{v}_0) = 0 \quad \forall~ F \in \Delta_{d-1}(\TT), ~ 0\in F,
\]  
then  
\[
\boldsymbol{\tau} = 0.
\]
\end{lemma}
\begin{proof}
Denote by \(\bs \tau_i = \boldsymbol{\tau}|_{T_i}(\texttt{v}_0)\) for \(i = 1, \ldots, d\). By assumption, \(\boldsymbol{\tau}_i \boldsymbol{n}_{F_i} = 0\). This reduces to the setting in Lemma~\ref{lem:Stensor1} by symbolically substituting \(F_i(T_0)\) with \(F_i(T)\) and \(\texttt{v}_c\) with \(\texttt{v}_0\). Consequently, we conclude that \(\boldsymbol{\tau}_i(\texttt{v}_0) = 0\) for \(i = 1, \ldots, d\), which implies \(\boldsymbol{\tau} = 0\).  
\end{proof}

\subsection{Linear element}

The function $\boldsymbol{\phi}_{ij}^{0}|_{T\setminus T_0}\in H(\div,T\setminus T_0; \mathbb S)$, but $\boldsymbol{\phi}_{ij}^{0}\not\in H(\div,T; \mathbb S)$, since $\boldsymbol{\phi}_{ij}^{0}\bs n_F|_F \neq 0$ for $F\in \Delta_{d-1}(\mathring{T}^{\rm R}) \cap \Delta_{d-1}(T_0)$, whereas $\boldsymbol{\phi}_{ij}^{0}|_{T_0} = 0$.
We incorporate \(\lambda_{0}^{\rm R}\) which vanishes on $F\in \Delta_{d-1}(\mathring{T}^{\rm R}) \cap \Delta_{d-1}(T_0)$.

\begin{lemma}\label{lem:phivijdivconformity}
For $1\leq i,j\leq d$, the function  
\[
\lambda_{0}^{\rm R} \boldsymbol{\phi}_{ij}^{0} = \lambda_{0}^{\rm R} \big(\chi_{T_i} \sym(\boldsymbol{t}_{0, c} \otimes \boldsymbol{t}_{0, j}) - \chi_{T_j} \sym(\boldsymbol{t}_{0, c} \otimes \boldsymbol{t}_{0, i})\big)  
\]  
belongs to \(H(\div, T; \mathbb{S}) \cap \mathbb{P}_1^{-1}(T^{\rm R}; \mathbb{S})\).
\end{lemma}

We then generalize the construction to any vertex, i.e., change index $0$ to $0\leq m\leq d$: 
\[
\Phi^{m}(\mathbb{S}) := \text{span} \left\{ \boldsymbol{\phi}_{ij}^{m} \mid i,j\in \{m\}^*, i < j \right\}.
\]
Let  
\begin{equation*}
\begin{aligned}
\Sigma_{1,\phi}^{\div}(T^{\rm R}; \mathbb{S}) ={} & V_1^{L}(T^{\rm R}; \mathbb{S}) + \sum_{m=0}^d \lambda_{m}^{\rm R}\Phi^{m}(\mathbb{S}) \\  
={} & \lambda_{c}^{\rm R}\mathbb P_0(T;\mathbb S) + \sum_{m=0}^d \lambda_{m}^{\rm R}\left [\mathbb P_0(T;\mathbb S) +\Phi^{m}(\mathbb{S})\right ].
\end{aligned}
\end{equation*}
As $\lambda_{m}^{\rm R}(\texttt{v}_c)=0$ for $m=0,1,\ldots, d$, any tensor-valued function in space $\Sigma_{1,\phi}^{\div}(T^{\rm R}; \mathbb{S})$ is single-value at $\texttt{v}_c$.

We will show these subspaces form a direct sum. It is well known that the dual basis of $V_1^{L}(T^{\rm R})$ is the nodal value at the vertices of $T^{\rm R}$. That is $\lambda_i^{\rm R}(\texttt{v}_j) = \delta_{i,j}$ for $i,j\in \{0, 1, \ldots, d, c\}$.

\begin{lemma}\label{lem:geodecP1}
We have
the geometric decomposition
\begin{equation}\label{eq:geodecP1}
\begin{aligned}
\Sigma_{1,\phi}^{\div}(T^{\rm R}; \mathbb{S}) ={} & V_1^{L}(T^{\rm R}; \mathbb{S}) \oplus \Oplus_{m=0}^d \lambda_{m}^{\rm R}\Phi^{m}(\mathbb{S}) \\  
={} & (\lambda_{c}^{\rm R}\mathbb P_0(T;\mathbb S)) \oplus \Oplus_{m=0}^d \big((\lambda_{m}^{\rm R}\mathbb P_0(T;\mathbb S)) \oplus \lambda_{m}^{\rm R}\Phi^{m}(\mathbb{S})\big).
\end{aligned}
\end{equation}
\end{lemma}
\begin{proof}

Assume
\begin{equation*}
\bs \sigma = \lambda_{c}^{\rm R}\boldsymbol{\tau}_c + \sum_{m=0}^d\lambda_{m}^{\rm R}\boldsymbol{\tau}_m =0,
\end{equation*}    
where $\boldsymbol{\tau}_c\in\mathbb P_0(T;\mathbb S)$, and $\boldsymbol{\tau}_m\in\mathbb P_0(T;\mathbb S)+\Phi^{m}(\mathbb{S})$. We will show $\bs \tau_c=\bs \tau_m = 0$. 

First of all, evaluate $\bs \sigma$ at $\texttt{v}_c$. As $\lambda_{m}^{\rm R}(\texttt{v}_c)=0$ for $m=0,1,\ldots, d$ and $\lambda_{c}^{\rm R}( \texttt{v}_c) = 1$, we get $\bs \tau_c = \bs \sigma(\texttt{v}_c) = 0$. 

Then evaluate $\bs \sigma$ at $\texttt{v}_0$, we get $\bs \tau_0(\texttt{v}_0) = 0$. Notice that $\bs \tau_0|_{T\setminus T_0}\in  H(\div,T\setminus T_0; \mathbb S)\cap \mathbb P_{0}^{-1}(T^{\rm R}\setminus T_0; \mathbb S)$. We can apply Lemma \ref{lem:vertexDoF} to conclude $\bs \tau_0|_{T\setminus T_0} = 0$. 

We expand $\bs \tau_0 = \bs \sigma_0 + \sum_{i,j=1}^d c_{ij}\boldsymbol{\phi}_{ij}^0$, $\bs \sigma_0\in \mathbb P_0(T;\mathbb S)$. In the expansion, we did not impose $i<j$, and thus the coefficient is skew-symmetric, i.e., $c_{ij} = -c_{ji}$. Consider the restriction $\bs \tau_0|_{T_j} = 0$ which implies 
\begin{equation*}
\boldsymbol{\sigma}_0 - 2\sum_{i=1, i\neq j}^dc_{ij}\sym(\boldsymbol{t}_{0, c} \otimes \boldsymbol{t}_{0, i})=0,\quad 1\leq j\leq d.
\end{equation*}
Multiplying the last equation by $\nabla\lambda_{j}$ from the right, and using the fact $\boldsymbol{t}_{0, c}\cdot\nabla\lambda_{j}=\frac{1}{d+1}$ and $\nabla \lambda_i\cdot \bs t_{0,j} = 0$, we obtain
\begin{equation*}
(d+1)\boldsymbol{\sigma}_0\nabla\lambda_{j} = \sum_{i\neq j}c_{ij}\boldsymbol{t}_{0, i},\quad 1\leq j\leq d.
\end{equation*}
By the duality \eqref{eq:duality0}, we multiply $(\nabla \lambda_i)^{\intercal}$ from the left to get
\begin{equation*}
c_{ij} = (d+1)(\nabla\lambda_i)^{\intercal}\boldsymbol{\sigma}_0\nabla\lambda_{j},\quad \forall~1\leq j,i\leq d.
\end{equation*}
Noting that $c_{ij}$ is skew-symmetric while $\boldsymbol{\sigma}_0$ is symmetric, so it is only possible that $\boldsymbol{\sigma}_0=0$ and $c_{ij}=0$ for all $1\leq i,j\leq d$.

Repeat this argument at every vertex $\texttt{v}_m$, we conclude $\bs \tau_m = 0$ for $m = 0, 1, \ldots, d$. 
\end{proof}

As a corollary, we can compute the dimension of the space:
\begin{equation*}
\dim \Sigma_{1,\phi}^{\div}(T^{\rm R}; \mathbb{S}) = \frac{1}{2}d(d+1)(2d+1), \qquad 15 (d =2), \quad 42 (d=3).
\end{equation*}
Next, we will give a basis of its dual space, i.e., degrees of freedom (DoF). 


\begin{lemma}\label{lem:P1Sunisolvent}
The space \(\Sigma_{1,\phi}^{\div}(T^{\rm R}; \mathbb{S})\) is uniquely determined by the degrees of freedom:
\begin{subequations}\label{HdivSfemk1dof}
\begin{align}
\boldsymbol{\tau} \boldsymbol{n}_F(\texttt{v}), & \quad \texttt{v} \in \Delta_0(T), \, F \in \Delta_{d-1}(T) \text{ containing }\texttt{v}, \label{HdivSfemk1dof1} \\
\boldsymbol{\tau}(\texttt{v}_c). &  \label{HdivSfemk1dof2}
\end{align}
\end{subequations}
\end{lemma}
\begin{proof}
The total number of DoFs in \eqref{HdivSfemk1dof} is given by
\begin{equation*}
(d+1)d^2 + \frac{1}{2}d(d+1) = \frac{1}{2}d(d+1)(2d+1),
\end{equation*}
which equals the dimension of \(\Sigma_{1,\phi}^{\div}(T^{\rm R}; \mathbb{S})\). 

To show the uni-solvence, take \(\boldsymbol{\tau} \in \Sigma_{1,\phi}^{\div}(T^{\rm R}; \mathbb{S})\) such that all the DoFs \eqref{HdivSfemk1dof} vanish. 
Thanks to the geometric decomposition \eqref{eq:geodecP1}, it follows from the vanishing of \eqref{HdivSfemk1dof2} that \(\boldsymbol{\tau} \in \Oplus_{m=0}^d \big((\lambda_{m}^{\rm R}\mathbb P_0(T;\mathbb S)) \oplus \lambda_{m}^{\rm R}\Phi^{m}(\mathbb{S})\big)\). Restricting to each vertex \(\texttt{v} \in \Delta_0(T)\), by Lemma \ref{lem:vertexDoF}, the vanishing of \eqref{HdivSfemk1dof1} implies \(\boldsymbol{\tau}(\texttt{v}) = 0\). Thus, \(\boldsymbol{\tau} = 0\).
\end{proof}

\begin{remark}\rm
As $\bs \tau \bs n_F|_F\in \mathbb P_1(F; \mathbb{R}^d)$, we can redistribute the DoFs \eqref{HdivSfemk1dof} to each face $F$~\cite[Example 3.1]{ChenHuang2024}, and obtain the following moment-based DoFs:
\begin{subequations}\label{HdivSfemk1momentdof}
\begin{align}
\int_F (\boldsymbol{\tau} \boldsymbol{n}_F) \cdot \boldsymbol{q} \, \mathrm{d}s, & \quad \boldsymbol{q} \in \mathbb{P}_1(F; \mathbb{R}^d), \, F \in \Delta_{d-1}(T), \label{HdivSfemk1momentdof1} \\
\int_T \boldsymbol{\tau} \, \mathrm{d}x. &  \label{HdivSfemk1momentdof2}
\end{align}
\end{subequations}
These moment-based DoFs are advantageous for interpolation and error analysis.
\end{remark}

Define the \(H(\div)\)-bubble polynomial space on $T$
\[
\mathbb{B}_1(\div, T; \mathbb{S}) = H_0(\div, T; \mathbb{S}) \cap \Sigma_{1,\phi}^{\div}(T^{\rm R}; \mathbb{S}),
\]
which, from the geometric decomposition \eqref{eq:geodecP1}, is characterized by
\begin{equation*}
\mathbb{B}_1(\div, T; \mathbb{S}) = \lambda_c^{\rm R}\mathbb P_0(T;\mathbb S).
\end{equation*}
That is the hat function on the barycenter is the only bubble polynomial in $\Sigma_{1,\phi}^{\div}(T^{\rm R}; \mathbb{S})$. The added sub-space $\sum_{m=0}^d \lambda_{m}^{\rm R}\Phi^{m}(\mathbb{S})$ is to add face-wise DoFs.

\begin{lemma}\label{lem:P1Sequiv}
It holds
\begin{equation*}
\Sigma_{1,\phi}^{\div}(T^{\rm R}; \mathbb{S})=H(\div, T; \mathbb{S}) \cap \mathbb{P}_1^{-1}(T^{\rm R}; \mathbb{S}).
\end{equation*}
\end{lemma}
\begin{proof}
By Lemma~\ref{lem:phivijdivconformity}, $\Sigma_{1,\phi}^{\div}(T^{\rm R}; \mathbb{S})\subseteq H(\div, T; \mathbb{S}) \cap \mathbb{P}_1^{-1}(T^{\rm R}; \mathbb{S})$. Then it suffices to show that 
DoFs \eqref{HdivSfemk1momentdof} are unisolvent for \(H(\div, T; \mathbb{S}) \cap \mathbb{P}_1^{-1}(T^{\rm R}; \mathbb{S})\).

Take $\boldsymbol{\tau}\in H(\div, T; \mathbb{S}) \cap \mathbb{P}_1^{-1}(T^{\rm R}; \mathbb{S})$ and assume that all the DoFs \eqref{HdivSfemk1momentdof} vanish. The vanishing DoF \eqref{HdivSfemk1momentdof1} indicates that $(\boldsymbol{\tau}\boldsymbol{n})|_{\partial T}=0$.
Apply Lemma~\ref{lem:vertexDoF} to obtain \(\boldsymbol{\tau}(\texttt{v}) = 0\) for each vertex \(\texttt{v} \in \Delta_0(T)\). Since $\boldsymbol{\tau}|_{F}\in\mathbb P_1(F;\mathbb S)$ for $F\in\Delta_{d-1}(T)$, we have $\boldsymbol{\tau}|_{\partial T}=0$, which implies $\bs \tau = \lambda^{\rm R}_c \bs \tau_0$ with some $\bs \tau_0\in \mathbb{P}_0^{-1}(T^{\rm R}; \mathbb{S})$. This together with the fact $\bs \tau\in H(\div, T; \mathbb{S})$ and Lemma~\ref{lem:Stensor1} means $\bs \tau_0\in H(\div,T; \mathbb S)\cap \mathbb P_{0}^{-1}(T^{\rm R}; \mathbb S)=\mathbb P_0(T;\mathbb S)$.
Thus, $\boldsymbol{\tau}=0$ holds from the vanishing DoF \eqref{HdivSfemk1momentdof2}.
\end{proof}

The div-conforming subspace $H(\div, T; \mathbb{S}) \cap \mathbb{P}_1^{-1}(T^{\rm R}; \mathbb{S})$ is defined by the constraint $[\boldsymbol{\sigma} \bs n]|_F = 0$ for all interior faces $F \in \Delta_{d-1}(\mathring{T}^{\rm R})$. Lemma \ref{lem:P1Sequiv} gives an explicit characterization for this subspace.

We also have the dimension identity
\[
\dim \mathbb{P}_1^{-1}(T^{\rm R}; \mathbb{S}) - \# \text{constraints} = \dim \Sigma_{1,\phi}^{\div}(T^{\rm R}; \mathbb{S}),
\]
where $ \# \text{constraints}$ is $d^2 = \dim \mathbb P_1(F; \mathbb R^d)$ for $[\bs \sigma\boldsymbol{n}]|_F$ on $d(d+1)/2$ interior faces. 

\subsection{Reduced linear element}
When considering the coarse mesh, the space \(\Sigma_{1,\phi}^{\div}(T^{\rm R}; \mathbb{S})\) can be further reduced by eliminating the interior DoF \eqref{HdivSfemk1dof2}. We define the reduced space of shape functions as  
\[
\begin{aligned}
\Sigma_{1,\phi}^{\div}(T; \mathbb S) :=& \{\boldsymbol{\tau} \in \Sigma_{1,\phi}^{\div}(T^{\rm R}; \mathbb{S}) : (\div \boldsymbol{\tau}, \boldsymbol{q})_T = 0 \text{ for } \boldsymbol{q} \in \mathbb{P}_1(T; \mathbb{R}^d)/\mathrm{RM}(T)\}\\
= & \{\boldsymbol{\tau} \in \Sigma_{1,\phi}^{\div}(T^{\rm R}; \mathbb{S}) : Q_{1,T}(\div \boldsymbol{\tau}) \in \mathrm{RM}(T)\}.
\end{aligned}
\]
As $\div \mathbb P_1(T; \mathbb S)\subset \mathrm{RM}(T)$, we have 
$ \mathbb P_1(T; \mathbb S)\subset \Sigma_{1,\phi}^{\div}(T; \mathbb S),$
which ensures the approximation property of $\Sigma_{1,\phi}^{\div}(T; \mathbb S)$. The interior basis $\lambda_c^{\rm R}\mathbb S$ can be used to impose the orthogonality or the range condition in the definition.

\begin{lemma}
The DoF \eqref{HdivSfemk1dof1} or \eqref{HdivSfemk1momentdof1} are unisolvent for \(\Sigma_{1,\phi}^{\div}(T; \mathbb S)\).
\end{lemma}

\begin{proof}
The number of DoF \eqref{HdivSfemk1dof1} or \eqref{HdivSfemk1momentdof1} is \(d^2(d+1)\), which does not exceed the dimension \(\dim \Sigma_{1,\phi}^{\div}(T; \mathbb S)\).

Suppose \(\boldsymbol{\tau} \in \Sigma_{1,\phi}^{\div}(T; \mathbb S)\) and the DoF \eqref{HdivSfemk1dof1} or \eqref{HdivSfemk1momentdof1} vanish. Then \((\boldsymbol{\tau} \boldsymbol{n})|_{\partial T} = 0\). Integration by parts yields  
\[
(\div \boldsymbol{\tau}, \boldsymbol{q})_T = 0, \quad \boldsymbol{q} \in \mathrm{RM}(T),
\]
and, by the definition of \(\Sigma_{1,\phi}^{\div}(T; \mathbb S)\),  
\[
(\div \boldsymbol{\tau}, \boldsymbol{q})_T = 0, \quad \boldsymbol{q} \in \mathbb{P}_1(T; \mathbb{R}^d).
\]
Thus, the DoF \eqref{HdivSfemk1momentdof2} also vanishes, and the proof of Lemma~\ref{lem:P1Sequiv} implies \(\boldsymbol{\tau} = 0\).
\end{proof}
The reduced finite element \(\Sigma_{1,\phi}^{\div}(T; \mathbb S)\) has less dimension 
$$
\dim \Sigma_{\mathrm{1,\phi}}^{\div}(T; \mathbb S) = d^2(d+1),\qquad 12 (d =2), \quad 36 (d=3).
$$

An additional reduction of \(\Sigma_{1,\phi}^{\div}(T; \mathbb S)\) is defined as  
\begin{equation}\label{eq:SigmaRM}
\begin{aligned}
\Sigma_{\mathrm{RM}}^{\div}(T; \mathbb S) &:= \{\boldsymbol{\tau} \in \Sigma_{1,\phi}^{\div}(T; \mathbb S) : \\
&\qquad\qquad (\boldsymbol{\tau} \boldsymbol{n})|_F \in (\mathbb{P}_1(F) \boldsymbol{n}_F \oplus \mathrm{RM}(F)) \text{ for } F \in \Delta_{d-1}(T)\},
\end{aligned}
\end{equation}
where the face rigid motion space is  
\[
\mathrm{RM}(F) = \mathbb{P}_0(F; \mathscr{T}^F) + \mathbb{P}_0(F; \mathbb{K}(\mathscr{T}^F)) \Pi_F \boldsymbol{x}.
\]
Here, \(\mathbb{K}(\mathscr{T}^F)\) denotes the space of skew-symmetric matrices on the tangential plane of \(F\). The dimension of this space is
$$
\dim \Sigma_{\mathrm{RM}}^{\div}(T; \mathbb S) = \frac{1}{2}d(d+1)^2,\qquad 9 (d =2), \quad 24 (d=3).
$$
The space $\Sigma_{\mathrm{RM}}^{\div}(T; \mathbb S)$ is uniquely determined by the DoFs  
\[
\int_F (\boldsymbol{\tau} \boldsymbol{n}) \cdot \boldsymbol{q} \, \mathrm{d}s, \quad \boldsymbol{q} \in (\mathbb{P}_1(F) \boldsymbol{n}_F \oplus \mathrm{RM}(F)), \, F \in \Delta_{d-1}(T).
\]

In Section \ref{sec:divlinear}, we will show
$$
\begin{aligned}
\div\Sigma_{1, \phi}^{\operatorname{div}}(\mathcal T_h; \mathbb{S}) = \div \Sigma_{\rm RM}^{\div}(\mathcal T_h; \mathbb S) &= {\rm RM}(\mathcal T_h),
\end{aligned}
$$
where
\begin{align*}
\Sigma_{1, \phi}^{\div}(\mathcal T_h; \mathbb S)&:=\{\boldsymbol{\tau}_h\in H(\div,\Omega;\mathbb S):  \boldsymbol{\tau}_h|_T\in \Sigma_{1,\phi}^{\div}(T; \mathbb S) \,\textrm{ for } T\in\mathcal T_h\}, \\
\Sigma_{\rm RM}^{\div}(\mathcal T_h; \mathbb S)&:=\{\boldsymbol{\tau}_h\in H(\div,\Omega;\mathbb S):  \boldsymbol{\tau}_h|_T\in \Sigma_{\rm RM}^{\div}(T; \mathbb S)\, \textrm{ for } T\in\mathcal T_h\}, \\
{\rm RM}(\mathcal T_h)&:=\{\boldsymbol{v}_h\in\mathbb P_{1}^{-1}(\mathcal T_h;\mathbb R^d): \boldsymbol{v}_h|_T\in {\rm RM}(T)\; \textrm{ for } T\in\mathcal T_h\}.
\end{align*}
The inf-sup conditions for the space pairs $\Sigma_{1, \phi}^{\div}(\mathcal T_h; \mathbb S)\times {\rm RM}(\mathcal T_h)$ and $\Sigma_{\rm RM}^{\div}(\mathcal T_h; \mathbb S)\times {\rm RM}(\mathcal T_h)$ are stated in Lemma~\ref{lem:infsupreducedP1}.

\begin{remark}\rm
Although the dimensions are matched, the reduced linear elements \(\Sigma_{1,\phi}^{\div}(T; \mathbb S)\) and \(\Sigma_{\mathrm{RM}}^{\div}(T; \mathbb S)\) differ from those in~\cite{ChristiansenHu2023,GopalakrishnanGuzmanLee2024} when restricted to $d=2,3$.
\end{remark}

\section{High Order Elements}\label{sec:highorderelement}
We first recall the construction of \(H(\div)\)-conforming symmetric finite elements using the \(t\)-\(n\) decomposition from~\cite{ChenHuang2024}. We then enrich the normal components to redistribute the degrees of freedom facewisely. By coupling with the bubble polynomials, we construct several \(H(\div)\)-conforming symmetric matrix elements.

\subsection{$t$-$n$ decomposition of symmetric tensor element}
We begin with the tensor product of the Lagrange element and the symmetric matrix \(\mathbb{S}\):  
\begin{align}
\label{eq:SLagrange}
\mathbb{P}_k(T; \mathbb{S})  &= \Oplus_{\ell = 0}^d \Oplus_{f \in \Delta_{\ell}(T)} \left[ b_f \mathbb{P}_{k - (\ell + 1)}(f) \otimes \mathbb{S} \right].
\end{align}

For sub-simplex \( f \in \Delta_{\ell}(T) \) with \( 0 \leq \ell \leq d \) and a linear space $V^f$ associated with $f$, define
\begin{equation}\label{eq:BkV}
\mathbb{B}_kV^f: =  b_{f}\mathbb{P}_{k-\ell-1}(f) \otimes V^f. 
\end{equation}
%
We will apply the notation $\mathbb{B}_kV^f$ to $V^f = \mathscr T^f(\mathbb S)$ and $\mathscr N^f(\mathbb S)$ in the $t$-$n$ decomposition~\eqref{Sdecomp}. Then \eqref{eq:SLagrange} can be rewritten as  
\begin{align}
\notag \mathbb P_k(T; \mathbb S)  &= \Oplus_{\ell = 0}^d\Oplus_{f\in \Delta_{\ell}(T)} \left [\mathbb B_k\mathscr T^f(\mathbb S)\oplus \mathbb B_k\mathscr N^f(\mathbb S) \right ] \\
\notag
&= \mathbb B_k(\div, T;\mathbb S) \oplus \Oplus_{\ell = 0}^{d-1}\Oplus_{f\in \Delta_{\ell}(T)} \mathbb B_k\mathscr N^f(\mathbb S),
\end{align}
where the div bubble polynomial space is, for $k\geq 2$,  
\begin{equation}\label{eq:divbubble}
\mathbb B_k(\div, T;\mathbb S) := \Oplus_{\ell = 1}^d\Oplus_{f\in \Delta_{\ell}(T)}\mathbb B_k\mathscr T^f(\mathbb S) = H_0(\div, T; \mathbb S)\cap \mathbb P_k(T; \mathbb S).
\end{equation}
The inclusion $\mathbb B_k(\div, T;\mathbb S) \subseteq H_0(\div, T; \mathbb S)\cap \mathbb P_k(T; \mathbb S)$ is relatively straightforward: for $F\supset f$, $\bs t^f\cdot \bs n_{F} = 0$, and for $F$ not containing $f$, $b_f|_{F} = 0$. The less trivial fact is $H_0(\div, T; \mathbb S)\cap \mathbb P_k(T; \mathbb S)\subseteq \mathbb B_k(\div, T;\mathbb S)$, which can be found in~\cite{ChenHuang2024}.  

The tangential-normal component $\sym ( \mathscr T^f\otimes \mathscr N^f)$ can be redistributed to $(d-1)$-dimensional faces by choosing the face normal basis $\{\bs n_{F_i}, i \in f^*\}$ and group the DoFs facewisely. The symmetric constraint in the component $\mathbb S( \mathscr N^f )$ are enforced with a global normal basis $\{\bs n_i^f\}$ that depends only on $f$. For the following DoFs, we note that for a vertex $f = \texttt{v}$, $\int_\texttt{v} u \dd s = u(\texttt{v})$, and $\mathbb P_k(\texttt{v}) = \mathbb R$ for any $k \geq 0$.  

\begin{theorem}[$H(\div; \mathbb S)$-conforming finite elements~\cite{ChenHuang2024}]  
\label{thm:femStensor}  
For each $f\in \Delta_{\ell}(\mathcal T_h)$, we choose a global $t$-$n$ basis $\{\bs t^f_1, \ldots, \bs t_{\ell}^f, \bs n^f_{1}, \ldots, \bs n^f_{d-\ell}\}$. Then the {\rm DoFs}  
\begin{subequations}\label{eq:divfemSkdof}
\begin{align}
\label{eq:divfemSkdof2}
\int_f((\bs n^f_i)^{\intercal}\boldsymbol{\tau}\bs n^f_j)\,{q}\dd s,
&\quad f\in \Delta_{\ell}(\mathcal{T}_h), {q}\in \mathbb P_{k-(\ell +1)}(f), \\
&\quad 1\leq i\leq j\leq d-\ell, \; \ell = 0,\ldots, d-1, \notag\\
\label{eq:divfemSkdof4}
\int_f ((\bs t^f_i)^{\intercal}\boldsymbol{\tau}\bs n_F)|_F\,{q}\dd s, &\quad F\in \Delta_{d-1}(\mathcal T_h), \; f\in \Delta_{\ell}(F), {q}\in \mathbb P_{k-(\ell +1)}(f),\\
&\quad i=1,\ldots,\ell, \; \ell = 1,\ldots, d-1, \notag\\
\label{eq:divfemSkdof5}
\int_T \boldsymbol{\tau}:\boldsymbol{q} \dx, &\quad   T\in \mathcal T_h, \; \bs q\in \mathbb B_k(\div, T;\mathbb S),
\end{align}
\end{subequations}
will determine a space $\Sigma_k^{\operatorname{div}}(\mathcal T_h;\mathbb S)\subset H(\div, \Omega; \mathbb S)$, where  
\begin{align*}
\Sigma_k^{\operatorname{div}}(\mathcal T_h;\mathbb S):=\{&\boldsymbol{\tau}\in L^2(\Omega;\mathbb S): \boldsymbol{\tau}|_T\in\mathbb P_k(T; \mathbb S) \quad\forall~T\in\mathcal T_h, \\
&\textrm{DoF \eqref{eq:divfemSkdof2} is single-valued across $f\in \Delta_{\ell}(\mathcal{T}_h)$ for $\ell=0,\ldots, d-1$}, \\
&\textrm{DoF \eqref{eq:divfemSkdof4} is single-valued across $F\in \Delta_{d-1}(\mathcal{T}_h)$} \}.
\end{align*}
\end{theorem}

Following~\cite{ChenChenHuangWei2024}, the moment DoFs can be replaced by function values at lattice points, and using the $t$-$n$ basis dual to that in DoFs, Lagrange type basis functions of $H(\div; \mathbb S)$-conforming finite elements can be derived.  

\begin{remark}\label{rm:tangential-normal}\rm
 DoF \eqref{eq:divfemSkdof4} can be further merged into one, leading to the modification in~\cite[Lemma 4.5]{ChenHuang2022}:  
\begin{equation}\label{eq:divStndof}
\int_F (\Pi_F\boldsymbol{\tau}\bs n_F)\cdot\boldsymbol{q}\dd s, \quad F\in \Delta_{d-1}(\mathcal T_h), \boldsymbol{q}\in \textrm{ND}_{k-2}(F),
\end{equation}
where $\textrm{ND}_{k-2}(F):=\{\boldsymbol{q}\in \mathbb P_{k-1}(F;\mathbb R^{d-1}): \boldsymbol{q}\cdot\boldsymbol{x}\in\mathbb P_{k-1}(F)\}$, and $\Pi_F$ is the projection of a vector onto the plane $\mathscr T^F$.  
\end{remark}

\subsection{Enrich the normal-normal component}
The tangential-normal components can be redistributed to each face $F\in \Delta_{d-1}(T)$ to get \eqref{eq:divStndof}. 
%
%
The normal-normal component $\{\bs n_{F_i} \otimes \bs n_{F_i}, i \in f^*\}$ can be redistributed to the face $F_i$. This redistribution can be used to construct a normal-normal continuous symmetric element~\cite{ChenHuang2025div-div-conforming,PechsteinSchoberl2011,PechsteinSchoeberl2018}.

For the off-diagonal components $\sym(\bs n_{F_i} \otimes \bs n_{F_j})$, these components can only be distributed to either $F_i$ or $F_j$, but not both. It can be redistributed to the edge $e_{ij}\in \Delta_{d-2}(T)$ and used to define hybridizable $H(\div\div)$-conforming finite elements; see our recent work~\cite{ChenHuang2025div-div-conforming}.

The symmetric constraint in $\mathbb S( \mathscr N^f )$ is imposed by choosing a global basis for the normal plane $\mathscr N^f$ for sub-simplices $f$ of dimension $0, 1, \ldots, d-2$ which introduces additional smoothness; see \eqref{eq:divfemSkdof2}. For example, all existing symmetric stress finite elements~\cite{ArnoldWinther2002,Adams;Cockburn:2005Finite,ArnoldAwanouWinther2008,HuZhang2016,Hu2015a,HuZhang2015,ChenHuang2022,ChenHuang2024,HuangZhangZhouZhu2024} are continuous at vertices. 

With the barycentric refinement, at a sub-simplex $f\in \Delta_{\ell}(T)$ with $\ell=0,1,\ldots, d-2$, for $i, j \in f^*$, we introduce  
\begin{equation*}
\boldsymbol{\phi}_{ij}^{f}:=\boldsymbol{\phi}_{ij}^{f(0)}= \chi_{T_{i}}\sym(\boldsymbol{t}_{f(0), c}\otimes\boldsymbol{t}_{f(0), j})-\chi_{T_{j}}\sym(\boldsymbol{t}_{f(0), c}\otimes\boldsymbol{t}_{f(0), i}).
\end{equation*}  
By Lemma~\ref{lem:phiijv}, $\boldsymbol{\phi}_{ij}^{f}|_{T\setminus T_{f(0)}} \in H(\div, T\setminus T_{f(0)}; \mathbb S)$. The vertex $f(0)$ can be replaced by any vertex of $f$, giving subspaces with the same trace on outer faces. 

\begin{figure}[htbp]
\subfigure[$t$-$n$ decomposition on $f$. The normal constraint is imposed by choosing a global normal plane basis.]{
\begin{minipage}[t]{0.45\linewidth}
\centering
\includegraphics*[width=3.2cm]{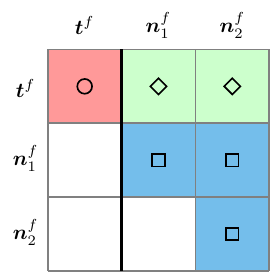}
\end{minipage}}
\qquad
\subfigure[$t$-$n$ decomposition on $f$ after the barycentric refinement which can be redistributed facewisely.]
{\begin{minipage}[t]{0.45\linewidth}
\centering
\includegraphics*[width=3.2cm]{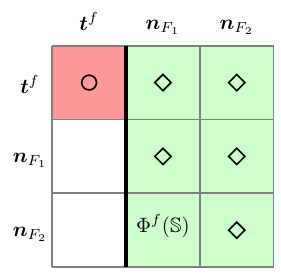}
\end{minipage}}
\caption{Red block ($\circ$): $H(\div)$-bubble polynomial basis. Green ($\diamond$): redistributed basis. Blue ($\scalebox{0.65}{$\square$}$): basis of $\mathbb S(\mathscr N^f)$ with a global normal plane basis. Here, $f$ is an edge of a tetrahedron, so $\dim \mathscr N^f = 2$. The vectors $\{\boldsymbol{n}_{1}^{f}, \boldsymbol{n}_{2}^{f}\}$ in (a) form a global basis of $\mathscr N^f$ used to impose the symmetric constraints. With the enrichment, we may instead use the face normals $\{\boldsymbol{n}_{F_{1}}, \boldsymbol{n}_{F_{2}}\}$, as shown in (b), so that the DoFs are redistributed to the two faces $F_1$ and $F_2$ containing $f$.}
\label{fig:divS1}
\end{figure}

We introduce 
$$
\Phi^f(\mathbb S) = {\rm span} \{ \boldsymbol{\phi}_{ij}^{f}: i,j\in f^*, 0\leq i<j\leq d\},
$$
and will enrich $\mathbb P_0(T;\mathscr N^f(\mathbb S))$ to
$$
\mathbb P_0(T;\mathscr N^f(\mathbb S)) + \Phi^f(\mathbb S).
$$
By Lemma~\ref{lem:geodecP1}, we have
$\mathbb P_0(T;\mathbb S) \cap \Phi^f(\mathbb S) = \{0\}$, and $\dim\Phi^f(\mathbb S) = \frac{1}{2}(d-\ell)(d-\ell-1)$ for $f\in\Delta_{\ell}(T)$ with $\ell=0,1,\ldots, d-2$.  
Then
\begin{equation*}
\dim\big(\mathbb P_0(T;\mathscr N^f(\mathbb S)) \oplus \Phi^f(\mathbb S)\big)=d(d-\ell).
\end{equation*}

We now have sufficient shape functions to redistribute the DoFs face-wise. See the illustration in Fig.~\ref{fig:divS1}. The following result is a generalization of Lemma \ref{lem:vertexDoF} from a vertex to a sub-simplex, and the proof is a generalization of that of Lemma \ref{lem:Stensor1}.

\begin{lemma}\label{lem:QNf}
Take $f \in \Delta_{\ell}(T)$ with $\ell = 0,1,\ldots, d-2$. For $\bs \tau \in \mathbb P_0(T;\mathscr N^f(\mathbb S)) \oplus \Phi^f(\mathbb{S})$, if 
\[
\boldsymbol{\tau}\boldsymbol{n}_{F_i}|_f = 0 \quad \forall\, i\in f^*,
\]
then $\boldsymbol{\tau} = 0$.
\end{lemma}

\begin{proof}
We choose scaled face-normal bases for $\mathcal{N}^f$ and its dual basis in $\mathcal{N}^f$:
\[
\{\nabla \lambda_i : i \in f^*\} \quad \text{and} \quad \{\hat{\bs n}_{F_i} : i \in f^*\},
\]
where
\[
\hat{\bs n}_{F_i} = \frac{\boldsymbol{n}^f_{f \cup \{i\}}}{\boldsymbol{n}^f_{f \cup \{i\}} \cdot \nabla \lambda_i} \quad \text{ and }\quad \hat{\bs n}_{F_i}\cdot \nabla\lambda_j = \delta_{i,j}, \, i,j\in f^*.
\]
Note that the tangential normal vector $\hat{\bs n}_{F_m} \in \mathscr{T}^F$ if $f \cup \{m\} \subseteq F$.

\begin{figure}[htbp]
\begin{center}
\includegraphics[width=4.15cm]{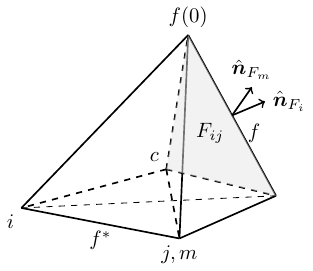}
\end{center}
\end{figure}

Let $\boldsymbol{\tau}_i = \boldsymbol{\tau}|_{T_i}\in\mathbb P_0(T_i;\mathbb S)$ for $i \in f^*$. Condition $(\boldsymbol{\tau}\boldsymbol{n}_{F_i})|_f = 0$ implies that 
$
\bs \tau_i \in \mathbb P_0(\mathscr T^{F_i}; \mathbb S).$
%
Let $Q_{\mathbb S(\mathcal{N}_f)}: \mathbb P_0(T_i;\mathbb S) \to\mathbb P_0(T_i;\mathbb S(\mathcal{N}_f))$ be the $L^2$ projection.
Then $Q_{\mathbb{S}(\mathcal{N}_f)}(\boldsymbol{\tau}_i)$ can be expressed as
\[
Q_{\mathbb{S}(\mathcal{N}_f)}(\boldsymbol{\tau}_i) = \sum_{m,n \in (f^* \cap \{i\}^*)} \tau_{i,m,n} \hat{\bs n}_{F_m} \otimes \hat{\bs n}_{F_n},
\]
where $\tau_{i,m,n} = (\nabla \lambda_m)^{\intercal} \bs \tau_i (\nabla \lambda_n)$, and by symmetry, $\tau_{i,m,n} = \tau_{i,n,m}$. As no normal component $\hat{\boldsymbol{n}}_{F_i}$ exists from $\bs \tau_i \in \mathbb P_0(\mathscr T^{F_i}; \mathbb S)$, the indices $m,n \in (f^* \cap \{i\}^*)$. 
Similarly, for $j\in f^*$, 
\[
Q_{\mathbb{S}(\mathcal{N}_f)}(\boldsymbol{\tau}_j) = \sum_{m,n \in (f^* \cap \{j\}^*)} \tau_{j,m,n} \hat{\bs n}_{F_m} \otimes \hat{\bs n}_{F_n},
\]
where $\tau_{j,m,n} = \tau_{j,n,m}$.

Let $F_{ij} = T_i \cap T_j$ and fix a unit normal vector $\boldsymbol{n}_{F_{ij}}$. 
Algebracially $F_{ij} = \{i, j\}^c$. For $f\in\Delta_{\ell}(F_{ij})$, the normal continuity $(\boldsymbol{\tau}_i\boldsymbol{n}_{F_{ij}})|_f = (\boldsymbol{\tau}_j\boldsymbol{n}_{F_{ij}})|_f$ implies that $$(Q_{\mathbb{S}(\mathcal{N}_f)}(\boldsymbol{\tau}_i)\boldsymbol{n}_{F_{ij}})|_f = (Q_{\mathbb{S}(\mathcal{N}_f)}(\boldsymbol{\tau}_j)\boldsymbol{n}_{F_{ij}})|_f.$$
It follows that
\begin{align*}
&\quad\;\tau_{i,j,j}(\hat{\bs n}_{F_j}\cdot\boldsymbol{n}_{F_{ij}})\hat{\bs n}_{F_j} + \sum_{m \in (f^* \cap \{i,j\}^c)} \tau_{i,m,j}(\hat{\bs n}_{F_j}\cdot\boldsymbol{n}_{F_{ij}})\hat{\bs n}_{F_m} \\
&= \tau_{j,i,i}(\hat{\bs n}_{F_i}\cdot\boldsymbol{n}_{F_{ij}})\hat{\bs n}_{F_i} + \sum_{m \in (f^* \cap \{i,j\}^c)} \tau_{j,m,i}(\hat{\bs n}_{F_i}\cdot\boldsymbol{n}_{F_{ij}})\hat{\bs n}_{F_m}.
\end{align*}
By comparing the coefficients in the basis $\{ \hat{\bs n}_{F_i}, \hat{\bs n}_{F_j}, \hat{\bs n}_{F_m}, m \in (f^* \cap \{i,j\}^c)\}$, we derive $\tau_{i,j,j}=\tau_{j,i,i}=0$ and the following relations among the coefficients:
\[
\tau_{i,m,j} \hat{\bs n}_{F_j} \cdot \boldsymbol{n}_{F_{ij}} = \tau_{j,m,i} \hat{\bs n}_{F_i} \cdot \boldsymbol{n}_{F_{ij}}, \quad \forall\, i,j,m \in f^*, \; i,j,m \text{ distinct}.
\]
Additionally, we expand the vector $\hat{\bs n}_{F_j}$ in $\mathscr T^{f\cup \{j\}}$ using basis $\{ \boldsymbol{t}_{j,s}, s\in f\}$
\[
\hat{\bs n}_{F_j} = \sum_{s\in f} (\hat{\bs n}_{F_j} \cdot \nabla \lambda_s) \boldsymbol{t}_{j,s},
\]
and since $\boldsymbol{t}_{c,s} \cdot \boldsymbol{n}_{F_{ij}} = 0$, it follows that
\[
\hat{\bs n}_{F_j} \cdot \boldsymbol{n}_{F_{ij}} = -(\boldsymbol{t}_{j,c} \cdot \boldsymbol{n}_{F_{ij}})(\hat{\bs n}_{F_j}\cdot \nabla \lambda_j).
\]
Substituting this into the relation for $\tau_{i,m,j}$, we get:
\[
\tau_{i,m,j}(\boldsymbol{t}_{j,c} \cdot \boldsymbol{n}_{F_{ij}}) = \tau_{j,m,i}(\boldsymbol{t}_{i,c} \cdot \boldsymbol{n}_{F_{ij}}), \quad \forall\, i,j,m \in f^*, \; i \neq j, \; i \neq m, \; j \neq m.
\]
Noting that $\boldsymbol{t}_{j,c} \cdot \boldsymbol{n}_{F_{ij}}=-\boldsymbol{t}_{i,c} \cdot \boldsymbol{n}_{F_{ij}}\neq0$, the last equation implies the skew-symmetric of the first two indices
\[
\tau_{i,j,m} = -\tau_{j,i,m}, \quad \forall\, i,j,m \in f^*, \; i \neq j, \; i \neq m, \; j \neq m.
\]

By symmetry and skew-symmetry of the coefficients $\tau_{i,j,m}$, it follows that
\[
\tau_{i,j,m} + \tau_{j,m,i} = 0, \quad \tau_{j,m,i} + \tau_{m,i,j} = 0, \quad \tau_{m,i,j} + \tau_{i,j,m} = 0.
\]
Solving these equations, we deduce $\tau_{i,j,m} = 0$ for all $i,j,m \in f^*$ with $i,j,m$ distinct. Consequently,
\begin{equation}\label{eq:20250521}
Q_{\mathbb{S}(\mathcal{N}_f)}(\boldsymbol{\tau}|_{T_i}) = 0 \quad \textrm{ for } i \in f^*.
\end{equation}

Next, we prove $\bs\tau=0$ following the argument of Lemma~\ref{lem:geodecP1}.
Expand $\bs \tau \in \mathbb P_0(T;\mathscr N^f(\mathbb S)) \oplus \Phi^f(\mathbb{S})$ as
\[
\bs \tau = \bs \sigma_0 + \sum_{m,j \in f^*} c_{mj}\boldsymbol{\phi}_{mj}^f,
\]
where $\bs \sigma_0\in \mathbb P_0(T;\mathscr N^f(\mathbb S))$ and the coefficients $c_{mj} \in \mathbb{R}$ satisfy the skew-symmetry condition: $c_{mj} = -c_{jm}$.
When restricted to the simplex $T_i$ with $i \in f^*$, this expansion becomes
\[
\bs \tau|_{T_i} = \bs \sigma_0 + 2\sum_{j \in  f^*\cap\{i\}^*} c_{ij}\sym(\boldsymbol{t}_{f(0), c}\otimes\boldsymbol{t}_{f(0), j}), \quad i \in f^*.
\]
Then multiplying $\nabla\lambda_j\otimes \nabla\lambda_i$ for $j \in  f^*\cap\{i\}^*$, by \eqref{eq:20250521}, it follows that
\[
c_{ij}=-(d+1)(\nabla\lambda_j)^{\intercal}\bs \sigma_0\nabla\lambda_i, \quad i,j \in f^*;\; i\neq j.
\]
By the symmetry of $\bs \tau$ and the skew-symmetry of $c_{ij}$, we have $c_{ij}=0$ for $i,j \in f^*$. Hence, $\bs \tau = \bs \sigma_0\in\mathbb P_0(T;\mathscr N^f(\mathbb S))$.
Finally, we conclude $\bs \tau = 0$ from the condition $(\bs \tau\boldsymbol{n}_{F})|_f = 0$ for all $F \in \Delta_{d-1}(T)$ with $f \subseteq F$.
\end{proof}

\subsection{Geometric decomposition}
Let \( f \in \Delta_{\ell}(T) \) with \( \ell = 0, 1, \ldots, d-2 \). Define the Lagrange bubble polynomial:
\[
b_f^{\rm R} = \lambda_{f(0)}^{\rm R} \lambda_{f(1)}^{\rm R} \cdots \lambda_{f(\ell)}^{\rm R}, \quad b_f = \lambda_{f(0)} \lambda_{f(1)} \cdots \lambda_{f(\ell)}.
\]
Notice that $b_f^{\rm R}|_f=b_f|_f$ but $b_f^{\rm R}$ and $b_f$ have different support.
By the fact $\boldsymbol{\phi}_{ij}^{f}|_{T\setminus T_{f(0)}}\in H(\div,T\setminus T_{f(0)}; \mathbb S)$, as Lemma~\ref{lem:phivijdivconformity}, we have the $H(\div,T)$-conformity of $b_{f}^{\rm R}\boldsymbol{\phi}_{ij}^{f}$.

\begin{lemma}
For each sub-simplex \( f \in \Delta_{\ell}(T) \) with \( 0 \leq \ell \leq d-2 \), and \( i, j \in f^*\), the function
\[
b_{f}^{\rm R}\boldsymbol{\phi}_{ij}^{f}=b_{f}^{\rm R}\big(\chi_{T_{i}} \sym(\boldsymbol{t}_{f(0), c} \otimes \boldsymbol{t}_{f(0), j}) - \chi_{T_{j}} \sym(\boldsymbol{t}_{f(0), c} \otimes \boldsymbol{t}_{f(0), i})\big)
\]
is in \( H(\div, T; \mathbb{S}) \cap \mathbb{P}_{\ell + 1}^{-1}(T^{\rm R}; \mathbb{S}) \).
\end{lemma}
We extend the notation $\mathbb B_kV^f$ by replacing $b_f$ in \eqref{eq:BkV} with $b_f^{\rm R}$:
\[
\mathbb{B}_k\Phi^f(\mathbb{S}) := b_{f}^{\rm R}\mathbb{P}_{k-\ell-1}(f) \otimes \Phi^f(\mathbb S) = \mathbb{P}_{k-\ell-1}(f) \otimes \operatorname{span}\left\{ b_{f}^{\rm R} \boldsymbol{\phi}_{ij}^f : i, j \in f^*, i < j \right\},
\]
and understand \( \mathbb{B}_k\Phi^F(\mathbb{S}) = \{0\} \) for \( F \in \Delta_{d-1}(T) \). Notice that for a given $k\geq 1$, $\mathbb{B}_k\Phi^f(\mathbb{S})$ is non-empty only for $\ell = \dim f \leq k-1$. In particular, for $k=1$ and $\ell = 0$, it is consistent with the space $\lambda^{\rm R}_m \Phi^m(\mathbb{S})$ for linear element.

We decompose the split mesh \( T^{\rm R} \) into sub-simplices:
\[
\Delta (T^{\rm R}) = \Oplus_{\ell = 0}^{d} \Delta_{\ell}(T^{\rm R}) = \Oplus_{\ell = 0}^{d-1}\left[\Delta_{\ell}(T) \cup \Delta_{\ell}(\mathring{T}^{\rm R})\right] + \Delta_{d}(T^{\rm R}).
\]
For interior sub-simplex, we use $\mathbb{B}_k\mathscr{N}^f(\mathbb{S})$ and further split $\mathscr{N}^f(\mathbb{S})$ by \eqref{Sdecomp}. The tangential-normal component will be redistributed facewisely. 

For sub-simplex $f\in \Delta_{\ell}(T)$, we enrich $\mathbb{B}_k\mathscr{N}^f(\mathbb{S})$ by $\mathbb{B}_k\Phi^f(\mathbb{S})$. So we introduce the space with the following geometric structure:
\begin{equation}\label{eq:localsigmadivk}
\begin{aligned}
\Sigma_{k,\phi}^{\operatorname{div}}(T^{\rm R}; \mathbb{S}) &= \Oplus_{\ell = 0}^{d-1} \bigg[\Oplus_{f\in \Delta_{\ell}(T)} \big( \mathbb{B}_k\mathscr{N}^f(\mathbb{S}) + \mathbb{B}_k\Phi^f(\mathbb{S}) \big) \\
&\quad\quad\quad\quad + \Oplus_{f\in \Delta_{\ell}(\mathring{T}^{\rm R})}\mathbb{B}_k\mathbb{S}(\mathscr{N}^f) \bigg] \\
&+ \Oplus_{F\in \Delta_{d-1}(\mathring{T}^{\rm R})}\Oplus_{\ell=0}^{d-1}\Oplus_{f\in \Delta_{\ell}(F)} \mathbb{B}_k\sym (\mathscr T^f\otimes \mathscr N^F)  \\
&+ \Oplus_{i = 0}^{d} \mathbb B_k(\div, T_i;\mathbb S).
\end{aligned}
\end{equation}

\begin{lemma}
All the sums in the definition of \( \Sigma_{k,\phi}^{\operatorname{div}}(T^{\rm R}; \mathbb{S}) \) are direct sums. The space \( \Sigma_{k,\phi}^{\operatorname{div}}(T^{\rm R}; \mathbb{S}) \) is a finite-dimensional subspace of \( H(\div, T^{\rm R}; \mathbb{S}) \), and its dimension is given by
\begin{equation}\label{eq:dimSigmaDivk}
\begin{aligned}
\dim\Sigma_{k,\phi}^{\operatorname{div}}(T^{\rm R}; \mathbb{S})&=d(d+1){k+d-1\choose d-1} + \frac{1}{2}d(d+1){k+d-2\choose d-1} \\
&+ \frac{1}{2}d(d+1)(k-1){k+d-2\choose d-2} + \frac{1}{2}d(d+1)^2{k+d-2\choose d} \\
&= \frac{1}{2}(d+1)\binom{k+d-1}{k}\bigl((d+1)k + d\bigr).
\end{aligned} 
\end{equation}
\end{lemma}
\begin{proof}
For the bubble polynomial $b_f^{\rm R}$, it satisfies
\begin{equation}\label{eq:bfRprop}
b_f^{\rm R}\mid _e = 0 \quad \forall e\in \Delta(T^{\rm R}), \dim e\leq \dim f, e\neq f.
\end{equation}
Then it suffices to show that the space \( \mathbb{B}_k\mathscr{N}^f(\mathbb{S}) + \mathbb{B}_k\Phi^f(\mathbb{S}) \) is a direct sum for each \( f \in \Delta_{\ell}(T) \) with \( 0 \leq \ell \leq d-2 \).

Assume
\begin{equation}\label{eq20250525}
\boldsymbol{\tau} + \sum_{i,j\in f^*} q_{ij}\boldsymbol{\phi}_{ij}^{f} = 0,
\end{equation}
where \( \boldsymbol{\tau} \in \mathbb{B}_k\mathscr{N}^f(\mathbb{S}) \) and \( q_{ij} \in b_f^{\rm R}\mathbb P_{k-(\ell +1)}(f) \).
Next, we prove that \( \boldsymbol{\tau}=0 \) and \( q_{ij}=0 \) for $i,j\in f^*$.

Notice that in the expansion, we include all $i,j\in f^*$ and thus $(q_{ij})$ is skew-symmetric, that is, $q_{ij}=-q_{ji}$ for $i,j\in f^*$.
Restrict \eqref{eq20250525} to simplex \(T_j\) for $j \notin f$ to get
\begin{equation*}
\boldsymbol{\tau} |_{T_{j}} + 2\sum_{i\in f^*}(q_{ji}|_{T_{j}})\sym(\boldsymbol{t}_{f(0), c} \otimes \boldsymbol{t}_{f(0), i})=0,\quad \forall~j\in f^*.
\end{equation*}
Then we follow the same argument as in Lemma~\ref{lem:geodecP1} to deduce that \( \boldsymbol{\tau}=0 \) and \( q_{ij}=0 \) for $i,j\in f^*$.
\end{proof}

\begin{remark}\rm
As an extension of Lemma \ref{lem:P1Sequiv} for $k=1$, we conjecture that 
\begin{equation*}
\Sigma_{k,\phi}^{\operatorname{div}}(T^{\rm R}; \mathbb{S}) = H(\div, T^{\rm R}; \mathbb{S})\cap \mathbb P_k^{-1} (T^{\rm R}; \mathbb{S}), \quad k\geq 1.
\end{equation*}
As necessary, the dimension identity holds
\[
\dim \mathbb{P}_k^{-1}(T^{\rm R}; \mathbb{S}) - \# \text{constraints} = \dim \Sigma_{k,\phi}^{\div}(T^{\rm R}; \mathbb{S}),
\]
where $ \# \text{constraints}$ is $d\binom{k+d-1}{k} = \dim \mathbb P_k(F; \mathbb R^d)$ for $[\bs \sigma\boldsymbol{n}]|_F$ on $d(d+1)/2$ interior faces. 
\end{remark}
\begin{theorem}[$H(\div; \mathbb{S})$-conforming composite finite elements]\label{thm:compositefemStensor}
For each \( f \in \Delta_{\ell}(T) \), we use \( \{\boldsymbol{n}_{F_i} : i \in f^*\} \) as the basis of \( \mathscr{N}^f \). For each \( f \in \Delta_{\ell}(\mathring{T}^{\rm R}) \), we choose a global \( t \)-\( n \) basis \( \{\boldsymbol{t}^f_1, \ldots, \boldsymbol{t}^f_{\ell}, \boldsymbol{n}^f_1, \ldots, \boldsymbol{n}^f_{d-\ell}\} \). Then the {\rm DoFs}
\begin{subequations}\label{eq:divfemRSkdof}
\begin{align}
\label{eq:divfemRSkdof1}
\int_f \boldsymbol{\tau} \boldsymbol{n}_F|_F \cdot \boldsymbol{q} \,\mathrm{d}s, &\quad f \in \Delta_{\ell}(T), \; \boldsymbol{q} \in \mathbb{P}_{k-(\ell+1)}(f; \mathbb{R}^d), \\
&\quad F \in \Delta_{d-1}(T), \; f \subseteq F, \; \ell = 0, \ldots, d-1, \notag \\
\label{eq:divfemRSkdof2}
\int_f \big((\boldsymbol{n}^f_i)^{\intercal} \boldsymbol{\tau} \boldsymbol{n}^f_j\big) q \,\mathrm{d}s, &\quad f \in \Delta_{\ell}(\mathring{T}^{\rm R}), \; q \in \mathbb{P}_{k-(\ell+1)}(f), \\
&\quad 1 \leq i \leq j \leq d-\ell, \; \ell = 0, \ldots, d-1, \notag \\
\label{eq:divfemRSkdof4}
\int_f \big((\boldsymbol{t}^f_i)^{\intercal} \boldsymbol{\tau} \boldsymbol{n}_F\big)|_F q \,\mathrm{d}s, &\quad F \in \Delta_{d-1}(\mathring{T}^{\rm R}), \;  f\in\Delta_{\ell}(F), \; q \in \mathbb{P}_{k-(\ell+1)}(f), \\
&\quad i = 1, \ldots, \ell, \; \ell = 1, \ldots, d-1, \notag \\
\label{eq:divfemRSkdof5}
\int_{T_i} \boldsymbol{\tau}:\boldsymbol{q} \dx, &\quad   T_i\in T^{\rm R}, \; \bs q\in \mathbb B_k(\div, T_i;\mathbb S),
\end{align}
\end{subequations}
will determine the space \( \Sigma_{k,\phi}^{\operatorname{div}}(T^{\rm R}; \mathbb{S}) \). 
\end{theorem}
\begin{proof}
Consider the DoF-Basis matrix sorted by the dimension of the sub-simplex. Due to the property \eqref{eq:bfRprop} of the bubble polynomial function $b_f^{\rm R}$, it is block diagonal. Thus, it suffices to consider one block, i.e., on one sub-simplex $f$ only. 
\begin{equation*}
\renewcommand{\arraystretch}{1.35}
\begin{array}{cc}
 &  
\begin{array}{ccccc}
\qquad \quad 0 \qquad& \quad 1 \qquad \quad& \ldots\quad\;\;\;	&  d-1\quad & \quad d\qquad\quad
\end{array}
\medskip
\\
\begin{array}{c}
0 \\ 1 \\ \vdots \\ d-1 \\ d
\end{array} 
& \left(
\begin{array}{>{\hfil$}m{1.2cm}<{$\hfil}|>{\hfil$}m{1.2cm}<{$\hfil}|>{\hfil$}m{1.2cm}<{$\hfil}|>{\hfil$}m{1.2cm}<{$\hfil}|>{\hfil$}m{1.2cm}<{$\hfil}}
\square & 0 & \cdots	& 0 & 0 \\
\hline
\square & \square & \cdots	& 0 & 0 \\
\hline
\vdots & \vdots & \ddots	& \vdots & \vdots \\
\hline
\square & \square & \cdots	& \square & 0 \\
\hline
\square & \square & \cdots	& \square& \square 
\end{array}
\right)
\end{array}
\end{equation*}

Consider $f\in \Delta_{\ell}(T)$ and $\boldsymbol{\tau}\in \mathbb B_k\mathscr N^f(\mathbb S) + \mathbb B_k\Phi^f(\mathbb S)+\mathbb B_k \mathscr T^f(\mathbb S)$. 
The vanishing DoF~\eqref{eq:divfemRSkdof1} implies 
$\bs \tau\bs n_F \mid_f = 0$. 
Expand $\bs \tau$ as
\begin{equation*}
\bs\tau=\bs\tau_0+\sum_{i=0}^{{k-1\choose\ell}}q_i(b_f\boldsymbol{\tau}_{1,i}+b_f^{\rm R}\boldsymbol{\tau}_{2,i}),
\end{equation*}
where $\bs\tau_0\in\mathbb B_k \mathscr T^f(\mathbb S)$, $\boldsymbol{\tau}_{1,i} \in \mathbb P_0(T;\mathscr N^f(\mathbb S))$, $\boldsymbol{\tau}_{2,i} \in \Phi^f(\mathbb{S})$, and $\{q_i: i=1,\ldots,{k-1\choose\ell} \}$ is a basis of space $\mathbb P_{k-(\ell+1)}(f)$.
By $\bs \tau\bs n_F \mid_f = 0$, we have
\begin{equation*}
\sum_{i=0}^{{k-1\choose\ell}}(q_i|_f)(\boldsymbol{\tau}_{1,i}+\boldsymbol{\tau}_{2,i})\bs n_F \mid_f=0.
\end{equation*}
This yields $(\boldsymbol{\tau}_{1,i}+\boldsymbol{\tau}_{2,i})\bs n_F |_f=0$ for $i=1,\ldots,{k-1\choose\ell}$.
By Lemma \ref{lem:QNf}, we obtain $\boldsymbol{\tau}_{1,i}=\boldsymbol{\tau}_{2,i}=0$, then $\bs \tau\in\mathbb B_k \mathscr T^f(\mathbb S)$. The vanishing DoF~\eqref{eq:divfemRSkdof5} implies $\bs \tau = 0$ by the characterization of div bubble polynomial \eqref{eq:divbubble}. 
\end{proof}

We then merge the DoFs and define the global finite element space. First, by the geometric decomposition of the Lagrange element, \eqref{eq:divfemRSkdof1} can be merged into \eqref{eq:globaldivfemRSkdof1}.  
Similarly, the interior tangential-normal components \eqref{eq:divfemRSkdof4} are merged into \eqref{eq:globaldivfemRSkdof4}.

\begin{theorem}
Let $\mathcal T_h^{\rm R}$ be the barycentric refinement of a triangulation $\mathcal T_h$. For each $f \in \Delta_{\ell}(\mathring{T}^{\rm R})$, select a global normal basis $\{\bs n^f_{1}, \ldots, \bs n^f_{d-\ell}\}$, where the vectors are linearly independent. The following DoFs:
\begin{subequations}\label{eq:globaldivfemRSkdof}
\begin{align}
\label{eq:globaldivfemRSkdof1}
\int_F \bs \tau \bs n_F\cdot \bs q \dd s, &\quad F\in \Delta_{d-1}(\mathcal T_h), \bs q\in \mathbb P_{k}(F;\mathbb R^d),\\
\label{eq:globaldivfemRSkdof2}
\int_f((\bs n^f_i)^{\intercal}\boldsymbol{\tau}\bs n^f_j)\,{q}\dd s,
&\quad f\in \Delta_{\ell}(\mathcal{\mathring{T}}_h^{\rm R}), {q}\in \mathbb P_{k-(\ell +1)}(f), \\
&\quad 1\leq i\leq j\leq d-\ell, \; \ell = 0,\ldots, d-1, \notag\\
\label{eq:globaldivfemRSkdof4}
\int_F (\Pi_F\boldsymbol{\tau}\bs n_F)\cdot\boldsymbol{q}\dd s, &\quad F\in \Delta_{d-1}(\mathcal{\mathring{T}}_h^{\rm R}), \boldsymbol{q}\in \mathrm{ND}_{k-2}(F),\\
\label{eq:globaldivfemRSkdof5}
\int_T\bs \tau : \bs q \dx, &\quad T\in \Delta_d(\mathcal T_h^{\rm R}), \; \bs q\in \mathbb B_{k}(\div,T;\mathbb S)
\end{align}
\end{subequations}
determine a space $\Sigma_{k,\phi}^{\operatorname{div}}(\mathcal T_h^{\rm R}; \mathbb{S}) \subset H(\div,\Omega; \mathbb S)$.
\end{theorem}

\subsection{Reduced finite element space}
If we are only interested in constructing a div-conforming finite element space on the original coarse mesh $\mathcal T_h$, the DoFs on $f$ interior to $T$ can be removed. Using another characterization of the div bubble polynomial~\cite{Hu2015a} 
$$\mathbb B_{k}(\div,T;\mathbb S) = \mathbb P_{k-2}(T; \mathbb S)\otimes {\rm span} \{\lambda_i\lambda_j \bs t_{ij}\otimes \bs t_{ij}, 0\leq i<j\leq d\},$$ the element-wise DoFs can be also simplified.

\begin{corollary}\label{cor:reducedspace}
For $k\geq2$, define the reduced finite element space: 
\begin{align*}
\Sigma_{k, \phi}^{\operatorname{div}}(\mathcal T_h; \mathbb{S}) &= \{\boldsymbol{\tau}_h\in H(\div,\Omega;\mathbb S): \boldsymbol{\tau}_h|_T\in \Sigma_{k, \phi}^{\operatorname{div}}(T; \mathbb{S}) \, \text{ for } T\in\mathcal T_h\},
\end{align*}
where the local space is defined as
\begin{equation*}
\begin{aligned}
\Sigma_{k, \phi}^{\operatorname{div}}(T; \mathbb{S}) &= \mathbb P_k(T; \mathbb S)\oplus \Oplus_{\ell = 0}^{d-2}\Oplus_{f\in \Delta_{\ell}(T)}\mathbb B_k\Phi^f(\mathbb S)\\
&= \mathbb B_{k}(\div, T;\mathbb S) \oplus \Oplus_{\ell = 0}^{d-1}\Oplus_{f\in \Delta_{\ell}(T)} \big( \mathbb{B}_k\mathscr{N}^f(\mathbb{S}) \oplus \mathbb{B}_k\Phi^f(\mathbb{S}) \big).    
\end{aligned}
\end{equation*}
Then 
\begin{equation*}
\dim\Sigma_{k, \phi}^{\operatorname{div}}(T; \mathbb{S}) =
\frac{1}{2}d(d+1){k+d\choose d} + \frac{1}{2}d(d+1){k+d-2\choose d-2},
\end{equation*}
and the following DoFs uniquely determine $\Sigma_{k, \phi}^{\operatorname{div}}(\mathcal T_h; \mathbb{S})$:
\begin{subequations}\label{eq:globaldivfemdof}
\begin{align}
\label{eq:globaldivfemdof1}
\int_F \bs \tau \bs n_F\cdot \bs q \dd s,&\quad F\in \Delta_{d-1}(\mathcal T_h), \bs q\in \mathbb P_{k}(F;\mathbb R^d),\\
\label{eq:globaldivfemdof5}
\int_T\bs \tau : \bs q \dx, &\quad T\in \Delta_d(\mathcal T_h), \; \bs q\in \mathbb P_{k-2}(T;\mathbb S). 
\end{align}
\end{subequations}
\end{corollary}
In view of the face DoF \eqref{eq:globaldivfemdof1}, the element $\Sigma_{k, \phi}^{\operatorname{div}}(\mathcal T_h; \mathbb{S})$ is the generalization of Brezzi-Douglas-Marini/N\'ed\'elec (2nd kind) div-conforming vector element~\cite{BrezziDouglasMarini1985,Nedelec:1986family,BrezziDouglasDuranFortin1987} to div-conforming symmetric stress element. Such a construction is not possible using $\mathbb P_k(T; \mathbb S)$ alone, but can be achieved by enriching it with $\mathbb B_k\Phi^f(\mathbb S)$.

Finally, by increasing the bubble space, a Raviart-Thomas (RT)-type element~\cite{RaviartThomas1977} with an enriched range can also be constructed. When the RT-type element is used to discretize the mixed elasticity problem, the approximation of the divergence of the discrete stress will be one order higher.
\begin{corollary}[RT-type element for symmetric tensors]
For $k\geq1$, the space of shape functions
$$
\Sigma_{k, \phi}^{\operatorname{div},+}(T; \mathbb{S}):=\mathbb B_{k+1}(\div, T;\mathbb S) \oplus \Oplus_{\ell = 0}^{d-1}\Oplus_{f\in \Delta_{\ell}(T)} \big( \mathbb{B}_k\mathscr{N}^f(\mathbb{S}) \oplus \mathbb{B}_k\Phi^f(\mathbb{S}) \big)
$$
is uniquely determined by the DoFs
\begin{align*}
\int_F \bs \tau \bs n_F\cdot \bs q \dd s,&\quad F\in \Delta_{d-1}(\mathcal T_h), \bs q\in \mathbb P_{k}(F;\mathbb R^d),\\
\int_T\bs \tau : \bs q \dx, &\quad T\in \Delta_d(\mathcal T_h), \; \bs q\in \mathbb P_{k-1}(T;\mathbb S). 
\end{align*}
\end{corollary}

The bubble space $\mathbb B_{k+1}(\div, T;\mathbb S)$ can be further reduced as no need to enrich $\ker(\div)$
$$
\mathbb B_{k}^+(\div, T;\mathbb S):=\{\boldsymbol{\tau}\in \mathbb B_{k+1}(\div, T;\mathbb S): \boldsymbol{\tau}\in \mathbb B_{k}(\div, T;\mathbb S) \textrm{ when }\div\boldsymbol{\tau}=0\}.
$$
When $k=1,2$, we have $\mathbb B_{k}^+(\div, T;\mathbb S)=\mathbb B_{k+1}(\div, T;\mathbb S)$.
As interior DoFs can be eliminated element-wise, such reduction is not necessary.

%

\section{Inf-sup Conditions}\label{sec:infsup}
In this section, we establish the inf-sup conditions on various $H(\div)$-conforming finite element spaces defined in the previous section.

\subsection{Existing inf-sup conditions}
For each $T \in \mathcal{T}_h$, the range of the divergence operator on the bubble space of symmetric tensors~\cite{Hu2015a,HuZhang2015} is
\begin{equation}\label{eq:SdivBrsurjection}
\div \mathbb{B}_k(\div, T; \mathbb{S}) = \mathbb{P}_{k-1}(T;\mathbb{R}^d) \cap {\rm RM}(T)^{\perp},
\end{equation}
where ${\rm RM}(T)^{\perp}$ is the $L^2$-orthogonal complement of ${\rm RM}(T)$ in $L^2(T; \mathbb{R}^d)$. To guarantee that the image of the divergence operator equals $\mathbb{P}_{k-1}^{-1}(\mathcal{T}_h;\mathbb{R}^d)$, by the div stability \eqref{eq:SdivBrsurjection} on bubble spaces, it suffices to include the following face degrees of freedom in order to handle the rigid motion space ${\rm RM}(T)$:
\begin{align*}
\int_F (\boldsymbol{\tau}\bs{n}_{F}) \cdot \boldsymbol{q} \dd s, \quad \boldsymbol{q} \in \mathbb{P}_{1}(F;\mathbb{R}^d), \quad F \in \Delta_{d-1}(\mathcal{T}_h).
\end{align*}
When $k \geq d+1$, the degrees of freedom \eqref{eq:divfemSkdof2} and \eqref{eq:divfemSkdof4} include these face terms, leading to the following inf-sup condition. 

\begin{lemma}[Proposition 4.10 in~\cite{ChenHuang2024}]\label{lem:Sdiscretedivinfsup}
Let $k \geq d+1$ and $\Sigma_k^{\operatorname{div}}(\mathcal{T}_h;\mathbb{S})$ be defined as in Theorem~\ref{thm:femStensor}. The following discrete inf-sup condition holds:
\begin{equation}\label{eq:Sdiscretedivinfsup}
\|\boldsymbol{v}_h\|_0 \lesssim \sup_{\boldsymbol{\tau}_h \in \Sigma_k^{\operatorname{div}}(\mathcal{T}_h;\mathbb{S})} \frac{(\div\boldsymbol{\tau}_h, \boldsymbol{v}_h)}{\|\boldsymbol{\tau}_h\|_0+\|\div\boldsymbol{\tau}_h\|_0} \quad \forall~\boldsymbol{v}_h \in \mathbb{P}_{k-1}^{-1}(\mathcal{T}_h;\mathbb{R}^d).
\end{equation}
\end{lemma}
\begin{proof}
Given $\boldsymbol{v}_h \in \mathbb{P}_{k-1}^{-1}(\mathcal{T}_h;\mathbb{R}^d)$, there exists $\boldsymbol{\tau} \in H^1(\Omega;\mathbb{S})$ such that (cf.~\cite{ArnoldHu2021})
\begin{equation}\label{eq:20221009-1}
\div\boldsymbol{\tau} = \boldsymbol{v}_h, \quad \|\boldsymbol{\tau}\|_1 \lesssim \|\boldsymbol{v}_h\|_0.
\end{equation}

Using the degrees of freedom \eqref{eq:divfemSkdof2} and \eqref{eq:divfemSkdof4}, construct $\widetilde{\boldsymbol{\tau}}_h \in \Sigma_k^{\operatorname{div}}(\mathcal{T}_h;\mathbb{S})$ such that
\begin{align*}
\int_F (\widetilde{\boldsymbol{\tau}}_h\bs{n}_{F}) \cdot \boldsymbol{q} \dd s &= \int_F (\boldsymbol{\tau}\bs{n}_{F}) \cdot \boldsymbol{q} \dd s, \quad \boldsymbol{q} \in \mathbb{P}_{1}(F;\mathbb{R}^d), \quad F \in \Delta_{d-1}(\mathcal{T}_h),
\end{align*}
and other degrees of freedom vanish. By a scaling argument,
\begin{equation}\label{eq:20221009-2}
\|\widetilde{\boldsymbol{\tau}}_h\|_0 + \|\div\widetilde{\boldsymbol{\tau}}_h\|_0 \lesssim \|\boldsymbol{\tau}\|_1 \lesssim \|\boldsymbol{v}_h\|_0.
\end{equation}

Next, integration by parts shows that $\div(\widetilde{\boldsymbol{\tau}}_h - \boldsymbol{\tau})|_T \in \mathbb{P}_{k-1}(T;\mathbb{R}^d) \cap {\rm RM}(T)^{\perp}$ for each $T \in \mathcal{T}_h$. By \eqref{eq:SdivBrsurjection}, there exists $\bs{b}_h \in L^2(\Omega;\mathbb{S})$ with $\bs{b}_h|_T \in \mathbb{B}_k(\div, T; \mathbb{S})$ such that
\begin{equation}\label{eq:20221009-3}
\div \bs{b}_h = \div(\boldsymbol{\tau} - \widetilde{\boldsymbol{\tau}}_h), \quad \|\bs{b}_h\|_{0,T} \lesssim h_T\|\div(\widetilde{\boldsymbol{\tau}}_h - \boldsymbol{\tau})\|_{0,T}.
\end{equation}

Define $\boldsymbol{\tau}_h = \widetilde{\boldsymbol{\tau}}_h + \bs{b}_h \in \Sigma_k^{\operatorname{div}}(\mathcal{T}_h;\mathbb{S})$. By \eqref{eq:20221009-1} and \eqref{eq:20221009-3},
\begin{equation}\label{eq:20221009-4}
\div \boldsymbol{\tau}_h = \div\widetilde{\boldsymbol{\tau}}_h + \div \bs{b}_h = \div\boldsymbol{\tau} = \boldsymbol{v}_h.
\end{equation}

Finally, from \eqref{eq:20221009-2} and \eqref{eq:20221009-3},
\begin{align}
\|\boldsymbol{\tau}_h\|_0+\|\div\boldsymbol{\tau}_h\|_0&=\|\boldsymbol{\tau}_h\|_0+\|\boldsymbol{v}_h\|_0\leq \|\widetilde{\boldsymbol{\tau}}_h\|_0+\|\bs b_h\|_0+\|\boldsymbol{v}_h\|_0 \notag\\
&\lesssim \|\widetilde{\boldsymbol{\tau}}_h\|_0+h\|\div\widetilde{\boldsymbol{\tau}}_h\|_{0}+\|\boldsymbol{v}_h\|_0\lesssim\|\boldsymbol{v}_h\|_0. \label{eq:20221009-5}
\end{align}
Combining \eqref{eq:20221009-4} and \eqref{eq:20221009-5} yields \eqref{eq:Sdiscretedivinfsup}.
\end{proof}

The requirement $k\geq d+1$ is restrictive, especially in high dimensions. We now consider the case $k = 2, \ldots, d$. By redistribution of tangential-normal DoFs facewisely (see Remark \ref{rm:tangential-normal}), we do have enough DoFs for the tangential-normal component
\begin{equation}\label{eq:tndofinfsup}
\int_F (\Pi_F\boldsymbol{\tau}\bs n_F)\cdot\boldsymbol{q}\dd s, \quad F\in \Delta_{d-1}(\mathcal T_h), \boldsymbol{q}\in \textrm{ND}_{k-2}(F).
\end{equation}
For $k\geq 2$, we have DoFs for  {\rm RM}($F$). To cover  {\rm RM}($T$), we only need to add normal-normal DoFs on faces. 
Following the notation \eqref{eq:BkV}, we introduce the space
\begin{equation*}
\mathbb B_{d+1}\mathbb S(\mathscr N^F) = {\rm span} \{ b_Fp\bs n_F\otimes \bs n_F, p\in \mathbb P_1(F)\} \subset \mathbb P_{d+1}(T;\mathbb S),
\end{equation*}
which can be determined by the DoFs
$$
\int_F \bs n_F^{\intercal}\bs \tau \bs n_F \, p\dd s, \quad p\in \mathbb P_1(F).
$$
Together with \eqref{eq:tndofinfsup}, we can ensure ${\rm RM}(T)$ is in the range of div operator.

Notice that the degree of $\div \mathbb B_{d+1}\mathbb S(\mathscr N^F)$ is higher than $k-1$ as $k\leq d$. Following~\cite{ChristHu2018Generalized,GuzmanNeilan2018Inf-sup}, we can modify the normal-normal bubble function to reduce the degree of its range.

\begin{lemma}
For any $\bs b^{nn} \in \mathbb B_{d+1}\mathbb S(\mathscr N^F)$, there exists a $\bs \beta^{nn} = {\rm Ext}(\bs b^{nn})\in \mathbb P_{d+1}(T; \mathbb S)$ such that  
$$
\bs \beta^{nn}\bs n \mid_{\partial T} = \bs b^{nn}\bs n \mid_{\partial T}, \quad \div \bs \beta^{nn} \in {\rm RM}(T),
$$
and
$$
\|\bs \beta^{nn}\|_{H(\div,T)} \lesssim \|\bs b^{nn}\|_{H(\div,T)}.
$$ 
\end{lemma}

\begin{proof}
Consider the function
$$
\boldsymbol{p} = (I - Q_{\rm RM})\div \bs b^{nn} \bot \, {\rm RM}(T).
$$
Then, by \eqref{eq:SdivBrsurjection}, we can find a bubble polynomial $\bs b_0 \in \mathbb B_{d+1}(\div,T; \mathbb S)$ such that $\div \bs b_0 = \boldsymbol{p}$ and $\bs b_0 \bs n \mid_{\partial T} = 0$. Let $\bs \beta^{nn} = \bs b^{nn} - \bs b_0$. Then $\bs \beta^{nn} \bs n \mid_{\partial T} = \bs b^{nn} \bs n \mid_{\partial T}$, and 
$$
\div \bs \beta^{nn} = \div \bs b^{nn} - \div \bs b_0 = Q_{\rm RM} \div \bs b^{nn} \in {\rm RM}(T).
$$
The stability follows from the scaling arguments.
\end{proof}

For $F\in \Delta_{d-1}(\mathcal T_h)$, for each $T$ containing $F$, we extend $\mathbb B_{d+1}\mathbb S(\mathscr N^F)$ to $T$ by using ${\rm Ext}$ operator elementwise. Then normal-normal components of the bubble space are well-defined across adjacent elements and maintain consistency within the mesh.

We now establish the following inf-sup condition by adding the normal-normal bubble functions. Similar enrichment strategies can be found in~\cite{HuangZhangZhouZhu2024}.

\begin{proposition}\label{prop:divSigmabar}
Let $\Sigma_{k, nn}^{\operatorname{div}}(\mathcal T_h; \mathbb{S}) = \Sigma_k^{\operatorname{div}}(\mathcal T_h; \mathbb{S}) + \Oplus_{F \in \Delta_{d-1}(\mathcal T_h)} {\rm Ext}(\mathbb B_{d+1}\mathbb S(\mathscr N^F))$. For $k \geq 2$, the divergence operator  
$$
\div: \Sigma_{k, nn}^{\operatorname{div}}(\mathcal T_h; \mathbb{S}) \to \mathbb P_{k-1}^{-1}(\mathcal T_h; \mathbb R^d)
$$
is surjective, and the following inf-sup condition holds:
\begin{equation*}
\|\boldsymbol{v}_h\|_0 \lesssim \sup_{\boldsymbol{\tau}_h \in \Sigma_{k, nn}^{\operatorname{div}}(\mathcal T_h; \mathbb{S})} \frac{(\div\boldsymbol{\tau}_h, \boldsymbol{v}_h)}{\|\boldsymbol{\tau}_h\|_0 + \|\div\boldsymbol{\tau}_h\|_0} \quad \forall~\boldsymbol{v}_h \in \mathbb P_{k-1}^{-1}(\mathcal T_h; \mathbb R^d).
\end{equation*}
\end{proposition}

As noted, the degrees of freedom \eqref{eq:globaldivfemRSkdof2} of $\bs n_i^{\intercal} \bs \tau \bs n_j$ on sub-simplices $f$ introduce constraints that prevent hybridization. 

\subsection{Inf-sup condition on the barycentric refinement}
After redistributing the degrees of freedom to the faces of the coarse element, we only need to add the normal-normal bubble functions for the interior faces to retain the inf-sup condition.

\begin{theorem}\label{th:divSigmabar}
Let
\begin{align*}
\Sigma_{k, \phi, nn}^{\operatorname{div}}(\mathcal T_h^{\rm R}; \mathbb{S}) &= \Sigma_{k,\phi}^{\operatorname{div}}(\mathcal T_h^{\rm R}; \mathbb{S}) + \Oplus_{F\in \Delta_{d-1}(\mathring{\mathcal T}_h^{\rm R})}{\rm Ext}(\mathbb B_{d+1}\mathbb S(\mathscr N^F)) \\
&= \{\boldsymbol{\tau}_h\in H(\div,\Omega;\mathbb S): \boldsymbol{\tau}_h|_T\in\Sigma_{k, \phi, nn}^{\operatorname{div}}(T^{\rm R}; \mathbb{S})\; \textrm{ for } T\in\mathcal{T}_h\},
\end{align*}
where
\begin{equation}\label{eq:localsigmadivknn}
\Sigma_{k, \phi, nn}^{\operatorname{div}}(T^{\rm R}; \mathbb{S}) = \Sigma_{k,\phi}^{\operatorname{div}}(T^{\rm R}; \mathbb{S}) + \Oplus_{F\in \Delta_{d-1}(\mathring{T}^{\rm R})}{\rm Ext}(\mathbb B_{d+1}\mathbb S(\mathscr N^F)).
\end{equation}
For $k \geq 2$, the divergence operator 
$$
\div:  \Sigma_{k, \phi, nn}^{\operatorname{div}}(\mathcal T_h^{\rm R}; \mathbb{S}) \to \mathbb P_{k-1}^{-1}(\mathcal T_h^{\rm R}; \mathbb R^d)
$$
is surjective, and the following inf-sup condition holds:
\begin{equation}\label{eq:infsupSigmabar}
\|\boldsymbol{v}_h\|_0 \lesssim \sup_{\boldsymbol{\tau}_h \in \Sigma_{k, \phi, nn}^{\operatorname{div}}(\mathcal T_h^{\rm R}; \mathbb{S})} \frac{(\div\boldsymbol{\tau}_h, \boldsymbol{v}_h)}{\|\boldsymbol{\tau}_h\|_0 + \|\div\boldsymbol{\tau}_h\|_0} \quad \forall~\boldsymbol{v}_h \in \mathbb P_{k-1}^{-1}(\mathcal T_h^{\rm R};\mathbb R^d).
\end{equation}
\end{theorem}

\begin{proof}
For a face $F \in \Delta_{d-1}(T)$ of an element $T \in \mathcal T_h$, we have the degrees of freedom $\int_F \bs \tau \bs n \cdot \bs q \,\dd s$ for $\bs q \in \mathbb P_2(F; \mathbb R^d)$. For interior faces within each coarse element, we have the additional degrees of freedom $\int_F \bs n^{\intercal} \bs \tau \bs n \, q \,\dd s$ for $q \in \mathbb P_1(F)$ and DoF \eqref{eq:tndofinfsup}. With these enriched DoFs, we can follow the proof of Lemma \ref{lem:Sdiscretedivinfsup} to construct a suitable $\tilde{\bs \tau}_h$ that satisfies the inf-sup condition. The remainder of the proof proceeds in the same manner as in Lemma \ref{lem:Sdiscretedivinfsup}.
\end{proof}

By \eqref{eq:dimSigmaDivk} and $\dim\mathbb B_{d+1}\mathbb S(\mathscr N^F)=d$, we have
\begin{equation}\label{eq:dimSigmaDivknn}
\begin{aligned}
&\dim\Sigma_{k,\phi,nn}^{\operatorname{div}}(T^{\rm R}; \mathbb{S})=\frac{(d+1)(kd+k+d)}{2}{k+d-1\choose d-1} +\frac{1}{2}d^2(d+1).
\end{aligned} 
\end{equation}

\subsection{Inf-sup condition on the coarse mesh}
%

If we work on the original coarse mesh, in view of DoF \eqref{eq:globaldivfemdof}, we have the following inf-sup condition. 

\begin{theorem}\label{th:infsupSigmahat}
Let $\Sigma_{k, \phi}^{\operatorname{div}}(\mathcal T_h; \mathbb{S})$ be the space defined in Corollary \ref{cor:reducedspace}. For $k \geq 2$,
$$
Q_{k-1,h}\div: \Sigma_{k, \phi}^{\operatorname{div}}(\mathcal T_h; \mathbb{S}) \to \mathbb P_{k-1}^{-1}(\mathcal T_h; \mathbb R^d)
$$ 
is surjective, and 
 \begin{equation}\label{eq:infsupSigmahat}
\|\boldsymbol{v}_h\|_0\lesssim \sup_{\boldsymbol{\tau}_h\in \Sigma_{k, \phi}^{\operatorname{div}}(\mathcal T_h; \mathbb{S})}\frac{(\div\boldsymbol{\tau}_h, \boldsymbol{v}_h)}{\|\boldsymbol{\tau}_h\|_0+\|\div\boldsymbol{\tau}_h\|_0}\quad\forall~\boldsymbol{v}_h\in \mathbb P_{k-1}^{-1}(\mathcal T_h;\mathbb R^d).
\end{equation}
\end{theorem}

We include the projection $Q_{k-1,h}$ to the coarse mesh since $\div \mathbb B_k\Phi^f$ is a space of piecewise polynomials on the split mesh $\mathcal T_h^{\rm R}$. We shall modify the shape function so that the range of the div operator is a polynomial on the coarse mesh. 
\begin{lemma}\label{lem:bubbleTR}
We have
 \begin{align*}
\div\left(\Sigma_{k, \phi, nn}^{\operatorname{div}}(T^{\rm R}; \mathbb{S}) \cap H_0(\div,T;\mathbb S)\right) = \mathbb P_{k-1}^{-1}(T^{\rm R};\mathbb R^d)\cap {\rm RM}(T)^{\bot}.
\end{align*}
\end{lemma}
\begin{proof}
Apply a similar proof as that for Theorem~\ref{th:divSigmabar} with $\bs \tau \in H_0(\div,T;\mathbb S)$.
\end{proof}

\begin{lemma}\label{lem:extendfacebubble}
For any $\bs \phi \in \mathbb B_k\Phi^f(\mathbb S)$, there exists $\bs \psi \in \Sigma_{k, \phi, nn}^{\operatorname{div}}(T^{\rm R}; \mathbb{S})$ s.t.
$$
\bs \psi\bs n \mid_{\partial T}= \bs \phi \bs n\mid_{\partial T}, \quad \div \bs \psi \mid_T \in  {\rm RM}(T)
$$ 
and 
$$
\|\bs \psi\|_{H(\div)}\lesssim \| \bs \phi\|_{H(\div)}.
$$ 
\end{lemma}
\begin{proof}
Consider the function
$$
\boldsymbol{p}\mid_T = (I - Q_{\rm RM})\div \bs \phi\mid_T \perp  {\rm RM}(T).
$$
As $\boldsymbol{p}\in \mathbb P_{k-1}^{-1}(T^{\rm R};\mathbb R^d)\cap {\rm RM}(T)^{\bot}$, by apply Lemma \ref{lem:bubbleTR}, we can find a bubble polynomial $\bs b_0\in \Sigma_{k, \phi, nn}^{\operatorname{div}}(T^{\rm R}; \mathbb{S})$ s.t. $\div \bs b_0 = \boldsymbol{p}$ and $\bs b_0\bs n\mid_{\partial T} = 0$. Let $\bs \psi = \bs \phi - \bs b_0$. The stability follows from the inf-sup condition \eqref{eq:infsupSigmabar}.
\end{proof}

For \( f \in \Delta_{\ell}(T) \) with \( 0 \leq \ell \leq d-2 \), we modify the added shape function space to $\mathbb{B}_k\Psi^f(\mathbb{S})$, which is defined as the space of all the extended functions $\bs \psi$ for $\bs \phi$ running over $\mathbb B_k\Phi^f(\mathbb S)$ in Lemma~\ref{lem:extendfacebubble}.

\begin{theorem}
For $k \geq 2$, let
\begin{align*}
\Sigma_{k, \psi}^{\operatorname{div}}(\mathcal T_h; \mathbb{S}) &= \{\boldsymbol{\tau}_h\in H(\div,\Omega;\mathbb S): \boldsymbol{\tau}_h|_T\in \Sigma_{k, \psi}^{\operatorname{div}}(T; \mathbb{S}) \quad \text{for } T\in\mathcal T_h\},
\end{align*}
where
\begin{equation}\label{eq:Sigmakpsi}
\Sigma_{k, \psi}^{\operatorname{div}}(T; \mathbb{S})=\mathbb P_k(T; \mathbb S)\oplus \Oplus_{\ell = 0}^{d-2}\Oplus_{f\in \Delta_{\ell}(T)}\mathbb B_k\Psi^f(\mathbb S).
\end{equation}
Then the operator
$$
\div: \Sigma_{k, \psi}^{\operatorname{div}}(\mathcal T_h; \mathbb{S}) \to \mathbb P_{k-1}^{-1}(\mathcal T_h; \mathbb R^d)
$$ 
is surjective, and 
 \begin{equation}\label{eq:infsupSigmakpsi}
\|\boldsymbol{v}_h\|_0\lesssim \sup_{\boldsymbol{\tau}_h\in \Sigma_{k, \psi}^{\operatorname{div}}(\mathcal T_h; \mathbb{S})}\frac{(\div\boldsymbol{\tau}_h, \boldsymbol{v}_h)}{\|\boldsymbol{\tau}_h\|_0+\|\div\boldsymbol{\tau}_h\|_0}\quad\forall~\boldsymbol{v}_h\in \mathbb P_{k-1}^{-1}(\mathcal T_h;\mathbb R^d).
\end{equation}
\end{theorem}
Clearly, $\Sigma_{k, \psi}^{\operatorname{div}}(T; \mathbb{S})$ shares the same DoFs as $\Sigma_{k, \phi}^{\operatorname{div}}(T; \mathbb{S})$, and
\begin{equation}\label{eq:dimSigmaDivkpsi}
\dim\Sigma_{k, \psi}^{\operatorname{div}}(T; \mathbb{S}) = \dim\Sigma_{k, \phi}^{\operatorname{div}}(T; \mathbb{S}) =
\frac{1}{2}d(d+1)\left({k+d\choose d} + {k+d-2\choose d-2}\right).
\end{equation}

\subsection{Linear element on the coarse mesh}\label{sec:divlinear}
We need $k\geq 2$ to include $\mathbb B_k(\div,T;\mathbb S)$ as no div bubble function for $k=1$. 
%
Now we consider the inf-sup condition on the coarse mesh for $k=1$. 

Define the global finite element space for symmetric tensors
\begin{align*}
\Sigma_{1, \phi}^{\operatorname{div}}(\mathcal T_h^{\rm R}; \mathbb{S})=\{\boldsymbol{\tau}_h\in L^2(\Omega;\mathbb S)&: \boldsymbol{\tau}_h|_T\in\Sigma_{1,\phi}^{\div}(T^{\rm R}; \mathbb{S}) \textrm{ for } T\in\mathcal T_h, \textrm{ DoF \eqref{HdivSfemk1dof1} or \eqref{HdivSfemk1momentdof1}}\\
&\, \textrm{ is single-valued across $(d-1)$-dimensional faces of $\mathcal T_h$}\}.
\end{align*}
By DoF \eqref{HdivSfemk1dof1} or \eqref{HdivSfemk1momentdof1}, $\Sigma_{1, \phi}^{\operatorname{div}}(\mathcal T_h^{\rm R}; \mathbb{S})$ is $H(\div)$-conforming.

\begin{lemma}
We have 
\begin{equation}\label{divontok1}
\|\boldsymbol{v}_h\|_0\lesssim \sup_{\boldsymbol{\tau}_h\in\Sigma_{1, \phi}^{\operatorname{div}}(\mathcal T_h^{\rm R}; \mathbb{S})}\frac{(\div\boldsymbol{\tau}_h, \boldsymbol{v}_h)}{\|\boldsymbol{\tau}_h\|_{H(\div)}} \quad\forall~\boldsymbol{v}_h\in\mathbb P_{1}^{-1}(\mathcal T_h;\mathbb R^d).
\end{equation}
That is 
$$
\begin{aligned}
Q_{1,h}\div\Sigma_{1, \phi}^{\operatorname{div}}(\mathcal T_h^{\rm R}; \mathbb{S}) &=\mathbb P_{1}^{-1}(\mathcal T_h;\mathbb R^d).
\end{aligned}
$$
Furthermore,
$$\div\Sigma_{1, \phi}^{\operatorname{div}}(\mathcal T_h^{\rm R}; \mathbb{S}) =\mathbb P_{0}^{-1}(\mathcal T_h^{\rm R};\mathbb R^d).$$
%
\end{lemma}
\begin{proof}
There exists a $\boldsymbol{\tau}\in H^1(\Omega;\mathbb S)$ satisfying 
\begin{equation*}
 \div\boldsymbol{\tau}=\boldsymbol{v}_h,\quad \|\boldsymbol{\tau}\|_1\lesssim \|\boldsymbol{v}_h\|_0.   
\end{equation*}
Let $\boldsymbol{\tau}_h\in\Sigma_{1, \phi}^{\operatorname{div}}(\mathcal T_h^{\rm R}; \mathbb{S})$ be the nodal interpolation based on DoFs \eqref{HdivSfemk1momentdof}. Then using the integration by parts and the scaling argument, we have
\begin{equation*}
(\div\boldsymbol{\tau}_h, \boldsymbol{v}_h)=(\div\boldsymbol{\tau}, \boldsymbol{v}_h)=\|\boldsymbol{v}_h\|_0^2,\quad \|\boldsymbol{\tau}_h\|_{H(\div)}\lesssim \|\boldsymbol{\tau}\|_1\lesssim \|\boldsymbol{v}_h\|_0.   
\end{equation*}
Therefore, \eqref{divontok1} holds.

The inf-sup condition \eqref{divontok1} implies 
$$
\dim\div\Sigma_{1, \phi}^{\operatorname{div}}(\mathcal T_h^{\rm R}; \mathbb{S})\geq\dim\mathbb P_{1}^{-1}(\mathcal T_h;\mathbb R^d)=\dim\mathbb P_{0}^{-1}(\mathcal T_h^{\rm R};\mathbb R^d).
$$
We end the proof by using the fact $\div\Sigma_{1, \phi}^{\operatorname{div}}(\mathcal T_h^{\rm R}; \mathbb{S})\subseteq\mathbb P_{0}^{-1}(\mathcal T_h^{\rm R};\mathbb R^d)$.
\end{proof}

\begin{lemma}
For $\boldsymbol{\tau}_h\in\Sigma_{1, \phi}^{\operatorname{div}}(\mathcal T_h^{\rm R}; \mathbb{S})$ satisfying $(\div\boldsymbol{\tau}_h, \boldsymbol{v}_h)=0$ for $\boldsymbol{v}_h\in\mathbb P_{1}^{-1}(\mathcal T_h;\mathbb R^d)$, we have $\div\boldsymbol{\tau}_h=0$. That is 
\begin{equation*}
\Sigma_{1, \phi}^{\operatorname{div}}(\mathcal T_h^{\rm R}; \mathbb{S})\cap\ker(Q_{1,h}\div)=\Sigma_{1, \phi}^{\operatorname{div}}(\mathcal T_h^{\rm R}; \mathbb{S})\cap\ker(\div).
\end{equation*}
\end{lemma}
\begin{proof}
Take $T\in\mathcal T_h$. Since $(\div\boldsymbol{\tau}_h)|_T\in\mathbb P_{0}^{-1}(T^{\rm R};\mathbb R^d)$, it suffices to prove $v=0$ for $v\in\mathbb P_{0}^{-1}(T^{\rm R})$ satisfying $(v, q)_T=0$ for $q\in\mathbb P_{1}(T)$. Choosing $q=\lambda_i$ with $i=0,1,\ldots, d$, we get
\begin{equation*}
\sum_{j=0}^dv_j\int_{T_j}\lambda_i\dx=0,
\end{equation*}
where $v_j=v|_{T_j}$. Hence,
$
v_i=\frac{d+2}{d+1}\sum_{j=0}^dv_j.
$ Therefore $v=0$.
\end{proof}

Similarly, we have the following inf-sup conditions for reduced linear elements.
\begin{lemma}\label{lem:infsupreducedP1}
We have the discrete inf-sup conditions
\begin{equation*}
\|\boldsymbol{v}_h\|_0\lesssim \sup_{\boldsymbol{\tau}_h\in\Sigma_{1, \phi}^{\div}(\mathcal T_h; \mathbb S)}\frac{(\div\boldsymbol{\tau}_h, \boldsymbol{v}_h)}{\|\boldsymbol{\tau}_h\|_{H(\div)}}\quad\forall~ \boldsymbol{v}_h\in{\rm RM}(\mathcal T_h),
\end{equation*}
\begin{equation*}
\|\boldsymbol{v}_h\|_0\lesssim \sup_{\boldsymbol{\tau}_h\in\Sigma_{\rm RM}^{\div}(\mathcal T_h; \mathbb S)}\frac{(\div\boldsymbol{\tau}_h, \boldsymbol{v}_h)}{\|\boldsymbol{\tau}_h\|_{H(\div)}}\quad\forall~\boldsymbol{v}_h\in{\rm RM}(\mathcal T_h).
\end{equation*}
\end{lemma}

\section{Discretization of Linear Elasticity Equation}\label{sec:discretization}
In this section, we apply the finite element spaces to the mixed formulation of the linear elasticity equation. With the established discrete inf-sup conditions, stability and error analysis follow directly. With enriched subspaces on the split mesh, all new stress elements in this work are hybridizable on the coarse mesh.

The linear elasticity problem can be written as the following first-order system:
\begin{equation}
\left\{
\begin{aligned}
\mathcal{A}\boldsymbol{\sigma}&=\boldsymbol{\varepsilon}(\boldsymbol{u}) \; \,\,\;\,\text{in} \;\Omega,
\\
\div\boldsymbol{\sigma}&=-\boldsymbol{f} \;\; \quad\text{in} \;\Omega ,
\\
\boldsymbol{u}&=\boldsymbol{0} \,\;\;\;\;\, \quad\text{on} \;\partial\Omega,\label{equ_elasticity}
\end{aligned}
\right.
\end{equation}
where,  $\mu$ and $\lambda$ are Lam\'e constants and $\lambda$ may be large, and 
$$
\mathcal{A}\boldsymbol{\sigma}=\frac{1}{2\mu}\boldsymbol{\sigma}-\frac{\lambda}{2\mu(2\mu+d\lambda)}\tr(\boldsymbol{\sigma})\boldsymbol{I} = \frac{1}{2\mu}\dev\boldsymbol{\sigma}+\frac{1}{d(2\mu+d\lambda)}\tr(\boldsymbol{\sigma})\boldsymbol{I}
$$
with
$\tr(\boldsymbol{\sigma})$ being the trace of tensor $\boldsymbol{\sigma}$, and
$\dev\boldsymbol{\sigma}:=\boldsymbol{\sigma} - \frac{1}{d}\tr(\boldsymbol{\sigma})\boldsymbol{I}.$

\subsection{Stabilized mixed method}
The stability of $Q_{k-1,h} \div$ established in Theorem \ref{th:infsupSigmahat} is weaker in the sense that $Q_{k-1,h} \div \bs \sigma = 0$ does not imply $\div \bs \sigma = 0$ pointwise, which may cause trouble in the discretization of the linear elasticity. We will address this issue by adding an element-wise stabilization.

Following~\cite{BrezziFortinMarini1993finite,ChenHuHuang2017Stabilized}, the stabilized variational form is: for $k\geq 2$, find $\boldsymbol{\sigma}_h\in \Sigma_{k, \phi}^{\operatorname{div}}(\mathcal T_h; \mathbb{S})$ and $\boldsymbol{u}_h\in \mathbb P_{k-1}^{-1}(\mathcal T_h; \mathbb R^d)$ such that
\begin{subequations}\label{elassdg}
\begin{align}
\label{elassdg1}
a_h(\boldsymbol{\sigma}_h,\boldsymbol{\tau}_h)+b_h(\boldsymbol{\tau}_h,\boldsymbol{u}_h)&=-(\boldsymbol{f},\div \boldsymbol{\tau}_h),\quad\;\forall~\boldsymbol{\tau}_h\in \Sigma_{k, \phi}^{\operatorname{div}}(\mathcal T_h; \mathbb{S}),
\\
\label{elassdg2}
b_h(\boldsymbol{\sigma}_h,\boldsymbol{v}_h)&=-(\boldsymbol{f},\boldsymbol{v}_h),\quad\quad\;\;\;\forall~\boldsymbol{v}_h\in \mathbb P_{k-1}^{-1}(\mathcal T_h; \mathbb R^d),
\end{align}
\end{subequations}
where
\begin{align*}
a_h(\boldsymbol{\sigma}_h,\boldsymbol{\tau}_h) &:=(\mathcal{A}\boldsymbol{\sigma}_h,\boldsymbol{\tau}_h) + \sum_{T\in \mathcal T_h}(\div \bs \sigma_h, \div \bs \tau_h)_T
,\\
b_h(\boldsymbol{\sigma}_h,\boldsymbol{u}_h)&:= \sum_{T\in\mathcal T_h}(\div \bs \sigma_h, \bs v_h)_T.
\end{align*}

The mixed finite method \eqref{elassdg} is well-posed, and possesses the following optimal error estimate. 

\begin{theorem}
The mixed finite element method \eqref{elassdg} for $k\geq 2$ is well-posed and stable. Let $(\boldsymbol{\sigma},\boldsymbol{u})\in H^{k+1}(\Omega; \mathbb{S})\times H^k(\Omega; \mathbb R^d)$ be the solution of problem \eqref{equ_elasticity}, and $(\boldsymbol{\sigma}_h,\boldsymbol{u}_h)\in\Sigma_{k, \phi}^{\operatorname{div}}(\mathcal T_h; \mathbb{S})\times\mathbb P_{k-1}^{-1}(\mathcal T_h; \mathbb R^d)$ be the solution of mixed method \eqref{elassdg}.
We have
\begin{equation}\label{eq:discreteerror}
\|\boldsymbol{\sigma}-\boldsymbol{\sigma}_h\|_{\div} + \|\boldsymbol{u}-\boldsymbol{u}_h\|_{0}\lesssim h^{k}(\|\boldsymbol{\sigma}\|_{k+1}+\|\boldsymbol{u}\|_{k}).
\end{equation}
\end{theorem}
\begin{proof}
We conclude the result from the inf-sup condition \eqref{eq:infsupSigmahat} and the robust coercivity~\cite{BoffiBrezziFortin2013,ChenHuHuang2018}
\begin{equation}\label{eq:discoercivity}
\|\bs \sigma_h\|_0^2 + \|\div \bs \sigma_h\|_0^2 \lesssim a_h(\bs \sigma_h, \bs \sigma_h), \quad \forall~\bs \sigma_h\in \Sigma_{k, \phi}^{\operatorname{div}}(\mathcal T_h; \mathbb{S}).
\end{equation}
The hidden constant in \eqref{eq:discoercivity} is independent of the Lam\'e constant $\lambda$. This implies that the estimate \eqref{eq:discreteerror} is robust with respect to the Lam\'e constant $\lambda$.
\end{proof}

\subsection{Hybridization}
We employ a hybridization technique~\cite{fraeijs1965displacement,ArnoldBrezzi1985} that relaxes the continuity conditions, which is applied on the coarse mesh $\mathcal T_h$, not on the split mesh $\mathcal T_h^{\rm R}$.

For $k\geq2$, introduce two discontinuous finite element spaces
\begin{align*}
\Sigma_{k, \psi}^{-1}(\mathcal T_h; \mathbb{S}):=&\{\boldsymbol{\tau}_h \in L^2(\Omega ; \mathbb{S}): \boldsymbol{\tau}_h|_T\in \Sigma_{k, \psi}^{\operatorname{div}}(T; \mathbb{S}) \textrm{ for } T\in\mathcal T_h\},\\
	\mathbb{P}_{k}^{-1}(\mathring{\mathcal F}_h; \mathbb{R}^{d}):=&\{\boldsymbol{\mu}_h \in L^2(\mathcal{F}_h;\mathbb{R}^{d}):\boldsymbol{\mu}_h|_F \in \mathbb{P}_{k}(F; \mathbb{R}^{d}) \text{ for each }F \in \mathring{\mathcal{F}}_h, \\
         &\qquad\qquad\qquad\qquad\;\text{ and } \boldsymbol{\mu}_h=0\text{ on } \mathcal{F}_h \backslash \mathring{\mathcal{F}}_h\},
\end{align*}
where $\mathcal{F}_h:=\Delta_{d-1}(\mathcal{T}_h)$ and $\mathring{\mathcal{F}}_h:=\Delta_{d-1}(\mathring{\mathcal{T}}_h)$.
The hybridization of the mixed finite element method \eqref{elassdg} is to find
$(\boldsymbol{\sigma}_h, \boldsymbol{u}_h, \boldsymbol{\lambda}_h) \in \Sigma_{k, \psi}^{-1}(\mathcal T_h; \mathbb{S}) \times \mathbb P_{k-1}^{-1}(\mathcal T_h; \mathbb R^d) \times \mathbb{P}_{k}^{-1}(\mathring{\mathcal F}_h; \mathbb{R}^{d})$
such that
\begin{subequations}\label{eq:hy}
\begin{align}
\label{hybridelassdg1}
a_h(\boldsymbol{\sigma}_h,\boldsymbol{\tau}_h)+b_h(\boldsymbol{\tau}_h,\boldsymbol{u}_h)+ c_h(\boldsymbol{\tau}_{h}, \boldsymbol{\lambda}_h)&=0,
\\
\label{hybridelassdg2}
b_h(\boldsymbol{\sigma}_h,\boldsymbol{v}_h)+ c_h(\boldsymbol{\sigma}_{h}, \boldsymbol{\mu}_h)&=-(\boldsymbol{f},\boldsymbol{v}_h)
\end{align}
\end{subequations}
for $(\boldsymbol{\tau}_h, \boldsymbol{v}_h, \boldsymbol{\mu}_h) \in \Sigma_{k, \psi}^{-1}(\mathcal T_h; \mathbb{S}) \times \mathbb P_{k-1}^{-1}(\mathcal T_h; \mathbb R^d) \times \mathbb{P}_{k}^{-1}(\mathring{\mathcal F}_h; \mathbb{R}^{d})$,
where $a_h(\boldsymbol{\sigma}_h,\boldsymbol{\tau}_h) =(\mathcal{A}\boldsymbol{\sigma}_h,\boldsymbol{\tau}_h)$, and the bilinear form
$
c_h(\boldsymbol{\tau}_{h}, \boldsymbol{\lambda}_h):=-\sum_{F \in \mathring{\mathcal{F}}_h}\left( [\boldsymbol{\tau}_h\boldsymbol{n}_F], \boldsymbol{\lambda}_h\right)_{F}
$ 
is introduced to impose the normal continuity. 
By relaxing the normal continuity across the interior faces, we can eliminate $\bs \sigma_h$ element-wise and obtain a symmetric and positive definite system. 

We follow our recent work~\cite{ChenHuang2025div-div-conforming} to introduce the weak div operator and establish the weak div stability.
Let $M_h:=\mathbb P_{k-1}^{-1}(\mathcal T_h; \mathbb R^d) \times \mathbb{P}_{k}^{-1}(\mathring{\mathcal F}_h; \mathbb{R}^{d})$. For $\boldsymbol{u}_h=(\boldsymbol{u}_0,\boldsymbol{u}_b), \boldsymbol{v}_h=(\boldsymbol{v}_0,\boldsymbol{v}_b)\in M_h$, introduce the inner product
\begin{equation*}
(\boldsymbol{u}_h, \boldsymbol{v}_h)_{0,h}=(\boldsymbol{u}_0, \boldsymbol{v}_0)+\sum_{F\in\mathcal F_h}h_F(\boldsymbol{u}_b, \boldsymbol{v}_b)_F,
\end{equation*}
which induces an $L^2$-type norm $\|\boldsymbol{v}_h\|_{0,h}=(\boldsymbol{v}_h,\boldsymbol{v}_h)_{0,h}^{1/2}$.
Define the weak div operator $\div_w: \Sigma_{k, \psi}^{-1}(\mathcal T_h; \mathbb{S})\to M_h$ by
\begin{equation*}
\div_w\boldsymbol{\tau}=\{\div(\boldsymbol{\tau}|_T), -h_F^{-1}[\boldsymbol{\tau}\boldsymbol{n}]\}_{T\in\mathcal{T}_h, F\in\mathring{\mathcal F}_h}.
\end{equation*}
Define norm
$$
\| \bs \tau \|_{\div_w} = \left (\|\boldsymbol{\tau}\|_0^2+\|\div_w\boldsymbol{\tau}\|_{0,h}^2\right )^{1/2}.
$$
\begin{theorem}
For $k \geq 2$, we have $\div_w\Sigma_{k, \psi}^{-1}(\mathcal T_h; \mathbb{S})=M_h$, and the discrete inf-sup condition 
 \begin{equation}\label{eq:weakdivinfsupSigmakpsi}
\|\boldsymbol{v}_h\|_{0,h}\lesssim \sup_{\boldsymbol{\tau}_h\in \Sigma_{k, \psi}^{-1}(\mathcal T_h; \mathbb{S})}\frac{(\div_w\boldsymbol{\tau}_h, \boldsymbol{v}_h)_{0,h}}{\| \bs \tau_h \|_{\div_w}}\quad\forall~\boldsymbol{v}_h=(\boldsymbol{v}_0,\boldsymbol{v}_b)\in M_h.
\end{equation}
\end{theorem}
\begin{proof}
First, choose $\boldsymbol{\tau}_b\in\Sigma_{k, \psi}^{-1}(\mathcal T_h; \mathbb{S})$ such that
\begin{equation*}
(\boldsymbol{\tau}_b|_T)\boldsymbol{n}=-\frac{1}{2}h_F\boldsymbol{v}_b\quad\textrm{ on face } F\in\Delta_{d-1}(T), T\in\mathcal{T}_h,
\end{equation*} 
and all the other DoFs vanish. Then
\begin{equation*}
\div_w\boldsymbol{\tau}_b=\{\div(\boldsymbol{\tau}_b|_T), \boldsymbol{v}_b\}_{T\in\mathcal{T}_h},\;\; \textrm{ and }\;\; \|\boldsymbol{\tau}_b\|_{\div_w}\lesssim \|\boldsymbol{v}_h\|_{0,h}.
\end{equation*}
By the inf-sup condition~\eqref{eq:infsupSigmakpsi}, there exists a $\boldsymbol{\tau}_0\in\Sigma_{k, \psi}^{\operatorname{div}}(\mathcal T_h; \mathbb{S})$ such that 
$$
\div_w \boldsymbol{\tau}_0 = \boldsymbol{v}_h - \div_w \boldsymbol{\tau}_b,
\quad 
\|\boldsymbol{\tau}_0\|_{0} + \|\div \boldsymbol{\tau}_0\|_{0} \lesssim 
\|\boldsymbol{v}_h - \div_w \boldsymbol{\tau}_b\|_{0,h} \lesssim \|\boldsymbol{v}_h\|_{0,h}.
$$
Setting $\boldsymbol{\tau}_h=\boldsymbol{\tau}_0 + \boldsymbol{\tau}_{b}$ yields $\div_w\boldsymbol{\tau}_h= \boldsymbol{v}_h$, and $\|\boldsymbol{\tau}_h\|_{\div_w} \lesssim \|\boldsymbol{v}_h\|_{0,h}$, which verifies the inf-sup condition~\eqref{eq:weakdivinfsupSigmakpsi}.
\end{proof}

%
%
%

Using the weak div operator, the hybridized mixed finite element method \eqref{eq:hy} can be rewritten as follows: find
$(\boldsymbol{\sigma}_h, \boldsymbol{u}_h) \in \Sigma_{k, \psi}^{-1}(\mathcal T_h; \mathbb{S}) \times M_h$
such that
\begin{align*}
(\mathcal{A}\boldsymbol{\sigma}_h,\boldsymbol{\tau}_h)+(\div_w\boldsymbol{\tau}_h, \boldsymbol{u}_h)_{0,h}&=0, \qquad\qquad\;\forall~\boldsymbol{\tau}_h\in \Sigma_{k, \psi}^{-1}(\mathcal T_h; \mathbb{S}),
\\
(\div_w\boldsymbol{\sigma}_h, \boldsymbol{v}_h)_{0,h}&=-(\boldsymbol{f},\boldsymbol{v}_0),\quad \forall~\boldsymbol{v}_h=(\boldsymbol{v}_0,\boldsymbol{v}_b)\in M_h.
\end{align*}

\begin{lemma}
We have the discrete coercivity
\begin{equation}\label{eq:hdgahcoercive}
\|\boldsymbol{\tau}_h\|_{0,h}^2\lesssim a_h(\boldsymbol{\tau}_h,\boldsymbol{\tau}_h),\quad\forall~\boldsymbol{\tau}_h\in Z_h,
\end{equation}
where
$$
Z_h:=\left\{ \boldsymbol{\tau}_h\in \Sigma_{k, \psi}^{-1}(\mathcal T_h; \mathbb{S}): \tr(\boldsymbol{\tau}_h)\in L_0^2(\Omega), \textrm{ and } \div_w\boldsymbol{\tau}_h = 0\right\}.
$$
\end{lemma}
\begin{proof}
By the definition of $\div_w\boldsymbol{\tau}_h$,
we find that $Z_h\subseteq \Sigma_{k, \psi}^{\operatorname{div}}(\mathcal T_h; \mathbb{S})\cap \ker(\div)$.
Thus, we end the proof by applying the coercivity~\eqref{eq:discoercivity}.
\end{proof}

Using the discrete inf-sup condition \eqref{eq:weakdivinfsupSigmakpsi} together with the discrete coercivity \eqref{eq:hdgahcoercive}, the well-posedness of the hybridized method \eqref{eq:hy} and its error estimates follow from standard arguments.

\begin{theorem}
 The hybridized formulation \eqref{eq:hy} is well-posed for $k\geq 2$. 
Let $\boldsymbol{\sigma}_h\in \Sigma_{k, \psi}^{-1}(\mathcal T_h; \mathbb{S})$ and $\boldsymbol{u}_h=\{\boldsymbol{u}_0,\boldsymbol{u}_b\}\in M_h$ be the solution of the hybridized formulation~\eqref{eq:hy}. Assume $\boldsymbol{\sigma}\in H^{k+1}(\Omega; \mathbb{S})$. We have $\boldsymbol{\sigma}_h\in\Sigma_{k, \psi}^{\operatorname{div}}(\mathcal T_h; \mathbb{S})$, and
\begin{equation*}
\|\boldsymbol{\sigma}-\boldsymbol{\sigma}_h\|_{0} + \|Q_{k-1,h}\div\boldsymbol{\sigma}-\div\boldsymbol{\sigma}_h\|_{0} + \|Q_{k-1,h}\boldsymbol{u}-\boldsymbol{u}_h\|_{0}\lesssim h^{k+1}\|\boldsymbol{\sigma}\|_{k+1}.
\end{equation*}
\end{theorem}

To derive a discrete $H^1$ error estimate for $\bs u_h$, we introduce the weak strain operator and establish another discrete inf-sup condition. 
Define $\boldsymbol{\varepsilon}_w: M_h\to\Sigma_{k, \psi}^{-1}(\mathcal T_h; \mathbb{S})$ as follows: for $\boldsymbol{v}_h=\{\boldsymbol{v}_0,\boldsymbol{v}_b\}\in M_h$, $\boldsymbol{\varepsilon}_w(\boldsymbol{v}_h)|_T\in \Sigma_{k, \psi}^{-1}(T; \mathbb{S})$ is determined by
\begin{align*}
(\boldsymbol{\varepsilon}_w(\boldsymbol{v}_h), \boldsymbol{\tau}_h)_T = - (\boldsymbol{v}_0, \div \boldsymbol{\tau}_h)_T + (\boldsymbol{v}_b, \boldsymbol{\tau}_h\boldsymbol{n})_{\partial T} \quad \forall~\boldsymbol{\tau}_h\in  \Sigma_{k, \psi}^{-1}(T; \mathbb{S}).
\end{align*}
Then by applying the integration by parts, we can easily show
$$(\boldsymbol{\varepsilon}_w(\boldsymbol{v}_h), \boldsymbol{\tau}_h) =\sum_{T\in\mathcal T_h}(\boldsymbol{\tau}_h, \boldsymbol{\varepsilon}(\boldsymbol{v}_0))_T + \sum_{T\in\mathcal T_h}(\boldsymbol{\tau}_h\boldsymbol{n}, \boldsymbol{v}_b-\boldsymbol{v}_0)_{\partial T}\quad\forall~\boldsymbol{\tau}_h\in\Sigma_{k, \psi}^{-1}(\mathcal T_h; \mathbb{S}),
$$
and the duality
\begin{equation*}
(\boldsymbol{\varepsilon}_w(\boldsymbol{v}_h), \boldsymbol{\tau}_h)  = - (\div_w\boldsymbol{\tau}_h, \boldsymbol{v}_h)_{0,h}\quad\forall~\boldsymbol{v}_h\in M_h, \boldsymbol{\tau}_h\in\Sigma_{k, \psi}^{-1}(\mathcal T_h; \mathbb{S}).
\end{equation*}

We respectively equip spaces $M_h$ and $\Sigma_{k, \psi}^{-1}(\mathcal T_h; \mathbb{S})$ with norms
\begin{align*}
\|\boldsymbol{v}_h\|_{1,h}^2&:=\sum_{T\in\mathcal{T}_h}\|\boldsymbol{\varepsilon}(\boldsymbol{v}_0)\|_T^2+\sum_{T\in\mathcal{T}_h}h_T^{-1}\|Q_{k,F}(\boldsymbol{v}_0-\boldsymbol{v}_b)\|_{\partial T}^2, \quad\;\; \boldsymbol{v}_h\in M_h, \\
\|\boldsymbol{\tau}_h\|_{0,h}^2&:=\|\boldsymbol{\tau}_h\|_0^2
+\sum_{T\in \mathcal{T}_h}h_T\|\boldsymbol{\tau}_h \boldsymbol{n}\|^2_{\partial T},
\qquad\qquad\qquad\qquad\quad\quad \boldsymbol{\tau}_h\in \Sigma_{k, \psi}^{-1}(\mathcal T_h; \mathbb{S}).
\end{align*}
It is easy to prove that
$\|\cdot\|_{1,h}$ is a norm on space $M_h$
and $\|\cdot\|_{0,h}$ is a norm on space $\Sigma_{k, \psi}^{-1}(\mathcal T_h; \mathbb{S})$ with $k\geq 2$.
Furthermore it is straightforward to verify the continuity
\begin{align*}
a_h(\boldsymbol{\sigma}_h,\boldsymbol{\tau}_h)&\lesssim \|\boldsymbol{\sigma}_h\|_{0,h}\|\boldsymbol{\tau}_h\|_{0,h}, \quad\forall~\boldsymbol{\sigma}_h,\boldsymbol{\tau}_h\in \Sigma_{k, \psi}^{-1}(\mathcal T_h; \mathbb{S}),
\\
(\boldsymbol{\varepsilon}_w(\boldsymbol{v}_h), \boldsymbol{\tau}_h)&\lesssim \|\boldsymbol{\tau}_h\|_{0,h}\|\boldsymbol{v}_h\|_{1,h}, \quad\forall~\boldsymbol{\tau}_h\in \Sigma_{k, \psi}^{-1}(\mathcal T_h; \mathbb{S}), \boldsymbol{v}_h\in M_h.
\end{align*}

\begin{lemma}
For $k\geq 2$, we have the discrete inf-sup condition
\begin{equation}\label{eq:hdgdiscinfsup}
\|\boldsymbol{v}_h\|_{1,h}\lesssim \sup_{\boldsymbol{\tau}_h\in\Sigma_{k, \psi}^{-1}(\mathcal T_h; \mathbb{S})}\frac{(\boldsymbol{\varepsilon}_w(\boldsymbol{v}_h), \boldsymbol{\tau}_h)}{\|\boldsymbol{\tau}_h\|_{0,h}}\quad\forall~\boldsymbol{v}_h\in M_h.
\end{equation}
\end{lemma}

\begin{proof}
Let $\boldsymbol{\tau}_h\in\Sigma_{k, \psi}^{-1}(\mathcal T_h; \mathbb{S})$ satisfy
\begin{align*}
(\boldsymbol{\tau}_h\boldsymbol{n},\boldsymbol{q})_F&=(h_F^{-1}Q_{k,F}(\boldsymbol{v}_b-\boldsymbol{v}_0), \boldsymbol{q})_F,\quad \forall~\boldsymbol{q}\in \mathbb P_k(F; \mathbb R^d), F\in\partial T, \\
(\boldsymbol{\tau}_h,\boldsymbol{q})_T&=(\boldsymbol{\varepsilon}( \boldsymbol{v}_0),\boldsymbol{q})_T, \qquad\qquad\quad\quad\;\, \forall~\boldsymbol{q}\in \boldsymbol{\varepsilon}(\mathbb P_{k-1}(T; \mathbb R^d)),
\end{align*}
and the rest DoFs vanish on each $T\in\mathcal{T}_h$. We have
\begin{equation*}
\|\boldsymbol{\tau}_h\|_{0,h}\lesssim \|\boldsymbol{v}_h\|_{1,h},
\end{equation*}
\begin{equation*}
(\boldsymbol{\varepsilon}_w(\boldsymbol{v}_h), \boldsymbol{\tau}_h) = \sum_{T\in\mathcal{T}_h}\|\boldsymbol{\varepsilon}(\boldsymbol{v}_0)\|_T^2+\sum_{T\in\mathcal{T}_h}h_T^{-1}\|Q_{k,F}(\boldsymbol{v}_0-\boldsymbol{v}_b)\|_{\partial T}^2=\|\boldsymbol{v}_h\|_{1,h}^2.
\end{equation*}
Therefore, the discrete inf-sup condition \eqref{eq:hdgdiscinfsup} follows.
\end{proof}

We follow the argument in~\cite{ChenHuHuang2018,LedererStenberg2023,LedererStenberg2024} to derive estimate \eqref{eq:energyestimate}, especially the superconvergence of $\|Q_h^M\boldsymbol{u}-\boldsymbol{u}_h\|_{1,h}$. 
The use of mesh-dependent norms in the analysis traces back to~\cite{BabuskaOsbornPitkaeranta1980} for the biharmonic equation,~\cite{PitkaerantaStenberg1983,Stenberg1986} for elasticity problems, and~\cite{LovadinaStenberg2006} for the Poisson equation.

\begin{theorem}\label{thm:hyerrestimate}
Let $\boldsymbol{\sigma}_h\in \Sigma_{k, \psi}^{-1}(\mathcal T_h; \mathbb{S})$ and $\boldsymbol{u}_h=\{\boldsymbol{u}_0,\boldsymbol{u}_b\}\in M_h$ be the solution of the hybridized formulation~\eqref{eq:hy} for $k\geq 2$. Assume $\boldsymbol{\sigma}\in H^{k+1}(\Omega; \mathbb{S})$. We have
\begin{equation}\label{eq:energyestimate}
\|\boldsymbol{\sigma}-\boldsymbol{\sigma}_h\|_{0} + \|Q_h^M\boldsymbol{u}-\boldsymbol{u}_h\|_{1,h}\lesssim h^{k+1}\|\boldsymbol{\sigma}\|_{k+1},
\end{equation}
where $Q_h^M\boldsymbol{u}=\{Q_{k-1,T}\boldsymbol{u}, Q_{k,F}\boldsymbol{u}\}_{T\in\mathcal{T}_h, F\in\Delta_{d-1}(\mathcal{T}_h)}$ is the $L^2$ projection of $\boldsymbol{u}$.
\end{theorem}
\begin{proof}
By applying the discrete inf-sup condition \eqref{eq:hdgdiscinfsup} and the discrete coercivity \eqref{eq:hdgahcoercive}, 
it holds the discrete stability
\begin{equation}\label{eq:hybridstability}
\begin{aligned}
&\quad\; \|\widetilde{\boldsymbol{\sigma}}_h\|_{0,h}+\|\widetilde{\boldsymbol{u}}_h\|_{1,h}\\
&\lesssim\sup_{\boldsymbol{\tau}_h\in \Sigma_{k, \psi}^{-1}(\mathcal T_h; \mathbb{S}),\boldsymbol{v}_h\in M_h}\frac{a_h(\widetilde{\boldsymbol{\sigma}}_h,\boldsymbol{\tau}_h)-(\boldsymbol{\tau}_h, \boldsymbol{\varepsilon}_w(\widetilde{\boldsymbol{u}}_h)) -(\widetilde{\boldsymbol{\sigma}}_h, \boldsymbol{\varepsilon}_w(\boldsymbol{v}_h))}{\|\boldsymbol{\tau}_h\|_{0,h}+\|\boldsymbol{v}_h\|_{1,h}}
\end{aligned}
\end{equation}
for any $\widetilde{\boldsymbol{\sigma}}_h\in \Sigma_{k, \psi}^{-1}(\mathcal T_h; \mathbb{S})$ and $\widetilde{\boldsymbol{u}}_h\in M_h$.

Take $\widetilde{\boldsymbol{\sigma}}_h=\boldsymbol{\sigma}_I-\boldsymbol{\sigma}_h$ and $\widetilde{\boldsymbol{u}}_h=Q_h^M\boldsymbol{u}-\boldsymbol{u}_h$, where $\boldsymbol{\sigma}_I\in\Sigma_{k, \psi}^{\operatorname{div}}(\mathcal T_h; \mathbb{S})$ is the nodal interpolation of $\boldsymbol{\sigma}$ based on DoFs \eqref{eq:globaldivfemdof}.
Employing the integration by parts and the definition of $\boldsymbol{\varepsilon}_w$, we have
\begin{align*}
(\boldsymbol{\tau}_h, \boldsymbol{\varepsilon}(\boldsymbol{u}))-(\boldsymbol{\tau}_h, \boldsymbol{\varepsilon}_w(Q_h^M\boldsymbol{u}))&=0, \qquad\forall~\boldsymbol{\tau}_h\in \Sigma_{k, \psi}^{-1}(\mathcal T_h; \mathbb{S}), \\
(\div\boldsymbol{\sigma}, \boldsymbol{v}_0)+(\boldsymbol{\sigma}_I, \boldsymbol{\varepsilon}_w(\boldsymbol{v}_h))&=0, \qquad\forall~\boldsymbol{v}_h=\{\boldsymbol{v}_0,\boldsymbol{v}_b\}\in M_h.
\end{align*}
Then
\begin{align*}
a_h(\widetilde{\boldsymbol{\sigma}}_h,\boldsymbol{\tau}_h)-(\boldsymbol{\tau}_h, \boldsymbol{\varepsilon}_w(\widetilde{\boldsymbol{u}}_h))= a_h(\boldsymbol{\sigma}_I,\boldsymbol{\tau}_h)-(\boldsymbol{\tau}_h, \boldsymbol{\varepsilon}_w(Q_h^M\boldsymbol{u})) =a_h(\boldsymbol{\sigma}_I-\boldsymbol{\sigma},\boldsymbol{\tau}_h),
\end{align*}
and 
\begin{equation*}
-(\widetilde{\boldsymbol{\sigma}}_h, \boldsymbol{\varepsilon}_w(\boldsymbol{v}_h))= (\boldsymbol{f},\boldsymbol{v}_0)-(\boldsymbol{\sigma}_I, \boldsymbol{\varepsilon}_w(\boldsymbol{v}_h))=0.
\end{equation*}
Now substituting the above two equations into \eqref{eq:hybridstability} gives
\begin{equation}\label{eq:superconvergence}
\|\boldsymbol{\sigma}_I-\boldsymbol{\sigma}_h\|_{0,h}+\|Q_h^M\boldsymbol{u}-\boldsymbol{u}_h\|_{1,h}\lesssim \|\boldsymbol{\sigma}-\boldsymbol{\sigma}_I\|_{0,h}.   
\end{equation}
Therefore, the estimate \eqref{eq:energyestimate} follows from the estimate of operator $Q_h^M$.
\end{proof}

Due to the exact divergence-free property, the error estimate \eqref{eq:superconvergence} depends only on $\|\boldsymbol{\sigma}-\boldsymbol{\sigma}_I\|$, independent of the error for $\boldsymbol{u}_h$. As a result, $\|Q_h^M\boldsymbol{u}-\boldsymbol{u}_h\|_{1,h}$ is one order higher than $\|\boldsymbol{u}-\boldsymbol{u}_h\|_{1,h}$, which is known as superconvergence.

\subsection{Postprocessing}

We will construct a new superconvergent approximation to the displacement $\boldsymbol{u}$ in virtue of the superconvergence $\|Q_h^M\boldsymbol{u}-\boldsymbol{u}_h\|_{1,h}$ in \eqref{eq:energyestimate}. 

Define a new approximation $\boldsymbol{u}_h^{\ast}\in \mathbb P_{k+1}^{-1}(\mathcal T_h; \mathbb R^d)$ to $\boldsymbol{u}$ piecewisely as a solution of the following problem: for each $T\in\mathcal{T}_h$,
\begin{subequations}\label{eq:postprocess}
\begin{align}
\label{eq:postprocess1}
(\boldsymbol{u}_h^{\ast}, \boldsymbol{q})_T&=((\boldsymbol{u}_h)_0, \boldsymbol{q})_T, \qquad\quad\, \forall~\boldsymbol{q}\in\textrm{RM}(T), \\
\label{eq:postprocess2}
(\boldsymbol{\varepsilon}(\boldsymbol{u}_h^{\ast}), \boldsymbol{\varepsilon}(\boldsymbol{q}))_T&=(\mathcal{A}\boldsymbol{\sigma}_h, \boldsymbol{\varepsilon}(\boldsymbol{q}))_T, \quad\quad\; \forall~\boldsymbol{q}\in\mathbb P_{k+1}(T;\mathbb R^d). 
\end{align}
\end{subequations}


\begin{theorem}\label{thm:uuhstar1}
Assume $\boldsymbol{\sigma}\in H^{k+1}(\Omega; \mathbb{S})$ and $\boldsymbol{u}\in H^{k+2}(\Omega;\mathbb R^d)$ for $k\geq 2$. Then
\[
\|\boldsymbol{\varepsilon}_h(\boldsymbol{u}-\boldsymbol{u}_h^{\ast})\|_{0}\lesssim h^{k+1}(\|\boldsymbol{\sigma}\|_{k+1}+\|\boldsymbol{u}\|_{k+2}),
\]
where $\boldsymbol{\varepsilon}_h$ is the elementwise strain operator with respect to $\mathcal T_h$.
\end{theorem}
\begin{proof}
Let $\boldsymbol{w}=Q_{k+1,h}\boldsymbol{u}-\boldsymbol{u}_h^{\ast}$ for simplicity.
It follows from \eqref{eq:postprocess2} with $\boldsymbol{v}=\boldsymbol{w}$ that
\[
(\boldsymbol{\varepsilon}(\boldsymbol{u}-\boldsymbol{u}_h^{\ast}), \boldsymbol{\varepsilon}(\boldsymbol{w}))_T=(\mathcal{A}(\boldsymbol{\sigma}-\boldsymbol{\sigma}_h), \boldsymbol{\varepsilon}(\boldsymbol{w}))_T.
\]
By the definition of $\boldsymbol{w}$,
\begin{align*}
\|\boldsymbol{\varepsilon}(\boldsymbol{w})\|_{0,T}^2 &= (\boldsymbol{\varepsilon}(Q_{k+1,h}\boldsymbol{u}-\boldsymbol{u}), \boldsymbol{\varepsilon}(\boldsymbol{w}))_T + (\mathcal{A}(\boldsymbol{\sigma}-\boldsymbol{\sigma}_h), \boldsymbol{\varepsilon}(\boldsymbol{w}))_T.
\end{align*}
This implies
\begin{equation*}
\|\boldsymbol{\varepsilon}(\boldsymbol{w})\|_{0,T}\lesssim \|\boldsymbol{\varepsilon}(Q_{k+1,h}\boldsymbol{u}-\boldsymbol{u})\|_{0,T} + \|\boldsymbol{\sigma}-\boldsymbol{\sigma}_h\|_{0,T}.
\end{equation*}
Finally, we end the proof by the triangle inequality and \eqref{eq:energyestimate}.
\end{proof}

\begin{remark}\rm
For the finite element pair $\Sigma_{1, \phi}^{\operatorname{div}}(\mathcal T_h^{\rm R}; \mathbb{S})\times \mathbb P_{1}^{-1}(\mathcal T_h;\mathbb R^d)$, by the discrete inf-sup condition \eqref{divontok1}, we have the following error estimate
\begin{equation*}
\|\boldsymbol{\sigma}-\boldsymbol{\sigma}_h\|_{0}  + \|\boldsymbol{u}-\boldsymbol{u}_h\|_{0}\lesssim h^2(\|\boldsymbol{\sigma}\|_{2}+\|\boldsymbol{u}\|_{2}).
\end{equation*}
For the finite element pairs $\Sigma_{1, \phi}^{\div}(\mathcal T_h; \mathbb S)\times {\rm RM}(\mathcal T_h)$ and $\Sigma_{\rm RM}^{\div}(\mathcal T_h; \mathbb S)\times {\rm RM}(\mathcal T_h)$, by the discrete inf-sup conditions in Lemma~\ref{lem:infsupreducedP1}, we have the following error estimates
\begin{equation*}
\|\boldsymbol{\sigma}-\boldsymbol{\sigma}_h\|_{0}  + \|\boldsymbol{u}-\boldsymbol{u}_h\|_{0}\lesssim h(\|\boldsymbol{\sigma}\|_{1}+\|\boldsymbol{u}\|_{1}).
\end{equation*}
\end{remark}

\bibliographystyle{abbrv}
\bibliography{./references,./paper}

\end{document}